\documentclass[12pt,reqno]{amsart}
\usepackage{amsmath,amsfonts,amssymb,amscd,amsthm,amsbsy,epsf}
\usepackage[dvips]{graphicx}
\usepackage{comment}
\usepackage[dvips]{color}
\numberwithin{equation}{section}
\newtheorem{theorem}{Theorem}
\newtheorem{meta-thm}[theorem]{Meta-Theorem}
\newtheorem{lemma}[theorem]{Lemma}
\newtheorem{cor}[theorem]{Corollary}
\newtheorem{proposition}[theorem]{Proposition}
\newtheorem{algorithm}[theorem]{Algorithm}
\newtheorem{remark}[theorem]{Remark}
\newtheorem{definition}[theorem]{Definition}
\newtheorem{conjecture}[theorem]{Conjecture}

\newcommand{\noaverage}[1]{({#1})^0}

\newcommand\beq[1]{ \begin{equation}\label{#1} }
\newcommand{\eeq}{ \end{equation} }

\newcommand\beqa[1]{ \begin{eqnarray} \label{#1}}

\newcommand{\eeqa}{ \end{eqnarray} }

\newcommand{\beqano}{ \begin{eqnarray*} }

\newcommand{\eeqano}{ \end{eqnarray*} }

\newcommand\equ[1]{{\rm (\ref{#1})}}

\def\ep{\varepsilon}

\def\Id{\operatorname{Id}}
\def\Im{\operatorname{Im}}

\def\A{{\mathcal A}}
\def\B{{\mathcal B}}
\def\C{{\mathcal C}}
\def\D{{\mathcal D}}

\def\F{{\mathcal F}}
\def\G{{\mathcal G}}
\def\H{{\mathcal H}}
\def\M{{\mathcal M}}
\def\R{{\mathcal R}}
\def\S{{\mathcal S}}
\def\eps{\varepsilon}
\def\complex{{\mathbb C}}
\def\integer{{\mathbb Z}}
\def\nat{{\mathbb N}}
\def\real{{\mathbb R}}
\def\torus{{\mathbb T}}

\begin{document}
\title[Domains of analyticity of KAM tori for dissipative
systems] {Domains of analyticity of Lindstedt expansions of KAM
tori in dissipative perturbations of Hamiltonian systems}

\author[R. Calleja]{Renato C.  Calleja}
\address{Department of Mathematics and Mechanics, IIMAS, National
  Autonomous University of Mexico (UNAM), Apdo. Postal 20-726,
  C.P. 01000, Mexico D.F., Mexico}
\email{calleja@mym.iimas.unam.mx}

\author[A. Celletti]{Alessandra Celletti}
\address{
Department of Mathematics, University of Roma Tor Vergata, Via della Ricerca Scientifica 1,
00133 Roma (Italy)}
\email{celletti@mat.uniroma2.it}

\author[R. de la Llave]{Rafael de la Llave}
\address{
School of Mathematics,
Georgia Institute of Technology,
686 Cherry St. Atlanta GA. 30332-1160 }
\email{rafael.delallave@math.gatech.edu}

\thanks{
A.C. was partially supported by PRIN-MIUR 2010JJ4KPA$\_$009, GNFM-INdAM and by the European Grant MC-ITN Astronet-II.
R.L. was partially supported by NSF grant DMS-1500943.
R.C. was partially supported by CONACYT grant 133036-F.
Part of this project was written while R.L. was visiting
the JLU-GT Joint Institute for Theoretical Science, and R.C. was visiting the
CRM (Barcelona) with a Llu\'is Santal\'o Visiting Position.}

\baselineskip=16pt              

\begin{abstract}
Many  problems in Physics are described by dynamical systems
that are conformally symplectic (e.g., mechanical systems with a friction proportional to
the velocity,
variational problems with a small discount or thermostated
systems). Conformally symplectic systems
are characterized by the property that they transform a symplectic form
into a multiple of itself.  The limit of small dissipation, which is the object of the
present study, is particularly interesting.

We provide all details for maps, but we present also the (somewhat minor) modifications needed
to obtain a direct proof for the case of differential equations.
We consider a family of conformally symplectic maps $f_{\mu, \eps}$ defined
on a $2d$-dimensional symplectic manifold $\M$ with exact symplectic form $\Omega$; we assume that $f_{\mu, \eps}$ satisfies
$f_{\mu, \eps}^* \Omega = \lambda(\eps) \Omega$.  We assume that
the family depends on a $d$-dimensional parameter $\mu$ (called \emph{drift}) and also on a small
scalar parameter $\eps$. Furthermore, we assume that the conformal factor $\lambda$
depends on $\eps$, in such a way that for $\eps=0$ we have
$\lambda(0)=1$ (the symplectic case).
We also assume that
$\lambda(\eps) = 1 + \alpha \eps^a + O(|\eps|^{a+1})$,
where $a\in\integer_+$, $\alpha\in\complex\backslash\{0\}$.

We study the domains of analyticity in $\eps$ near $\eps = 0$ of perturbative
expansions (Lindstedt series) of the parameterization of the quasi--periodic orbits of
frequency $\omega$ (assumed to be Diophantine)
 and of the parameter $\mu$. Notice that this is a singular perturbation, since
any friction (no matter how small) reduces the set of quasi-periodic
solutions in the system.
We prove that the Lindstedt series
are analytic in a domain in the complex $\eps$ plane, which is
obtained by taking from a ball centered at zero a sequence of smaller
balls with center along smooth lines going through the origin. The
radii of the excluded balls decrease faster than any power of the distance of
the center to the origin.  We state also a conjecture on the
optimality of our results.

The proof is based on the following procedure.  To find a
quasi-periodic solution, one solves an invariance equation for the
embedding of the torus, depending on the parameters of the
family. Assuming that the frequency of the torus satisfies a
Diophantine condition, under mild non--degeneracy assumptions, using a
Lindstedt procedure we construct an approximate solution to all orders
of the invariance equation describing the KAM torus; the zeroth order Lindstedt
series is provided by the solution of the invariance equation of the symplectic case. Starting from
such approximate solution, we implement an {\sl a-posteriori} KAM
theorem to get the true solution of the invariance equation, and we
show that the procedure converges. This allows also the study of
monogenic and Withney extensions.

\end{abstract}


\maketitle

\section{Introduction}
Many problems of physical interest are described by models
given by Hamiltonian systems with a small dissipation.

Some important particular cases of systems with dissipation are the following:
\begin{itemize}
\item
Hamiltonian systems
with a   dissipation proportional to the velocity describing, e.g.,
problems of Celestial Mechanics -- see \cite{Celletti2010};
\item
Euler-Lagrange equations of exponentially discounted
systems -- very natural in finance, when inflation is
present and one wants to
minimize the cost in present money
 --  see \cite{Bensoussan88, IturriagaS11, DFIZ2014};
\item
Gaussian thermostats - see \cite{DettmannM96, WojtkowskiL98}.
\end{itemize}

In the examples above, it was discovered that there is  a nice
geometric structure,
namely that the natural symplectic form is transformed
into a multiple of itself by the dynamics.
Systems that have this geometric property are referred
to as \emph{conformally symplectic}.
Besides their applications to physical problems, conformally
symplectic systems were considered on their own in differential geometry
(\cite{Banyaga02, Agrachev}).

This geometric structure has important consequences for
the dynamics (see Section~\ref{sec:conformallysymplectic}).
Notably for our purposes, there is a KAM theory
with an a-posteriori format and a systematic
way of obtaining perturbative expansions, see \cite{CallejaCL11}.

In this paper, we study the analyticity properties of
the parameterization of the quasi-periodic solutions
and of the drift parameter of conformally symplectic
systems. Perturbative expansions to all orders
are easy to obtain and have been considered for
a long time (as a matter of fact, we will develop also very efficient
methods of computation of perturbative expansions).

Notice that adding a dissipation to a Hamiltonian system is a very singular
perturbation. We expect that a Hamiltonian system admits quasi-periodic
solutions with many frequencies. On the other hand, a system with a
positive dissipation -- even if extremely small -- leads to  the
creation of attractors which have few -- or even none ! -- quasi-periodic
solutions. This is why one has to consider external parameters such as the drift.

Hence, we do not expect that the asymptotic expansions converge and
therefore the parameterization functions will not be analytic in $\eps$ in
any ball centered at the origin.  A common strategy in the literature
of singularly perturbed problems has been to device asymptotic
expansions and show that, even if they do not converge, they can be
resummed. See \cite{Balser, Hardy, BakerGM96} for general mathematical
treatments and \cite{LeGZJ90,ArtecaFC90,KleinertSF01} for surveys of
applications of resummation techniques
to concrete problems in Physics.  For singular
perturbation expansions in dynamical problems similar to the ones in this paper, there
have been quite a number of studies (see, e.g.,
\cite{GentileM96,GallavottiG02,GentileBD06,Gentile10,
  GentileG05,CorsiFG13,CorsiFG14}).
In this work, we follow a different approach from resummation. By using an a-posteriori theorem, we
show (see Theorem~\ref{main})
that the parameterization and the drift are analytic in
$\eps$ in a domain (called  $\G$ and defined precisely
in \eqref{goodset}) obtained by removing from a ball  centered at the origin a
sequence of (much smaller)  balls with centers in smooth curves going through
the origin. The radii of these balls decrease very fast (faster
than any power of the size of the distance to the center of the excluded ball) as the centers of
the excluded balls go to zero. Hence, even if the domain does not contain  a ball
centered at the origin, it is hard to distinguish it from a full ball (see Figure~\ref{fig:N0}).

We emphasize that we only prove rigorously that $\G$ is a lower bound
for the analyticity domain. The results can be improved in many
ways (see Section~\ref{sec:improved}). On the other hand, in
Section~\ref{sec:optimal} we conjecture that the domain $\G$ is
essentially optimal in the sense that for a generic system none of the
excluded balls can be filled completely (we present heuristic arguments
for the conjecture and a proof of a weaker result (Proposition~\ref{destruction})
 that shows that, given any ball,
one can find systems where the analyticity domain does not cover the ball). We note
that the domains we find have only sectors
centered at the origin with very small apertures, say $\pi/a$, where
$a$ is the leading exponent in the change of the conformal factor.

The conjecture that there are no sectors of analyticity with aperture bigger
than $\pi/a$ has consequences
for resummability methods and asymptotic expansions, since uniqueness of the function given by
the expansion and Borel summation often requires that
there are sectors of angular aperture of about $\pi$, see \cite{Sokal}
(the well known Cauchy
functions $\exp(-1/\eps^2)$ provide examples of non-trivial functions
with trivial asymptotic expansions in domains of angular aperture less
than $\pi$).

We also recall that,
since the parameterization is a function, the notion of analytic functions
taking values on parameterizations
requires that we specify a Banach space in which the parameterization
lies.  This is done in Section~~\ref{sec:definitions}. We anticipate
that these spaces for the parameterization are spaces of analytic functions
from a complex extension of the torus to the phase space.

The method of proof of this paper is to formulate a functional
equation for the parameterization of the torus and the drift
expressing that the torus
is  invariant for the map with the adjusted parameters.
First, we show (part A of
Theorem~\ref{mainKAM}) that it is possible to find a solution of the
invariance equation in the sense of formal power series (this is a
very standard order by order perturbation expansion, but in
Section~\ref{fastlindstedt} we also construct fast algorithms
that double the order in each step).  By
truncating appropriately this formal power series, we obtain functions
that solve the invariance equation very approximately.  These
approximate solutions are taken as the initial points of an iterative
procedure which is shown to converge by an a-posteriori theorem (see
Theorem~\ref{mainKAM}), which is very similar to Theorem 20 in
\cite{CallejaCL11}. The domain $\G$ is obtained by examining carefully
the process, and the quantitative and explicit conditions of
Theorem~\ref{mainKAM}.
Notice that for the applications in this paper,
it is essential
that Theorem~\ref{mainKAM} is formulated in an a-posteriori format
(we can start the iteration of an approximate solution
even if the problem is not close
to integrable and also we can obtain estimates on the distance from
the initial data to the solution by the error of the initial approximation).

Methods similar to those described above (using a formal power
series as a jumping off point of an a-posteriori method) were used in
\cite{JorbaLZ00}. One can also mention \cite{CallejaCL13,CCCL2015} which deal with
dissipative systems, although
in the two latter papers the iterative method is not a Newton's method,
but rather a contraction argument (taking advantage of the fact that the
dissipation is very strong).
The method is very well suited for the study of monogenic properties and
Whitney extensions, which we study in Section~\ref{fastlindstedt}.

We expect that the phenomena uncovered here also hold in several other problems.

\subsection{Conformally symplectic systems}
\label{sec:conformallysymplectic}
As mentioned above, there are many physical problems that lead to the study of conformally symplectic
systems and, in particular, to the study of singular series.

One simple but very useful remark for conformally symplectic systems
is that quasi-periodic solutions with enough frequencies satisfy the
so-called \emph{automatic reducibility}: in a neighborhood of
an (approximately) invariant torus, one can find a system of coordinates
in which the linearization becomes (approximately) constant coefficients.
Automatic reducibility happens irrespective of whether the system is close to
integrable or not (see \cite{CallejaCL11}).  Automatic reducibility
allows one to develop a KAM theory (\cite{CallejaCL11}), leading also
to efficient and accurate algorithms.

In \cite{CallejaCL11} one can find a
numerically accessible method to compute
the breakdown threshold extending the arguments of \cite{CallejaL10}.
This criterion for breakdown, based on the study of the growth of
Sobolev norms, is shown to converge to
 the right value of the threshold; indeed, the accuracy of the computed breakdown
 in actual computers is only limited by the available memory, precision and computational time.

A numerical implementation in actual computers was done
in \cite{CallejaC10}, which includes  comparisons with
other methods. Another property of KAM tori in conformally symplectic
systems is that the breakdown of the invariant circles
happens due to a \emph{bundle collapse}
scenario in which hyperbolicity is lost because
the stable bundle becomes close to the tangent bundle,
even if the exponents remain bounded away
(see \cite{CallejaF11} for a numerical implementation and a presentation of
empirical results, including scaling relations for the breakdown).

We also mention that the Greene's criterion for the breakdown of invariant circles has been
extended to conformally symplectic systems and given a partial
justification (see \cite{CCFL14}).

It is known that KAM theory for conformally symplectic systems
requires adjusting parameters (\cite{CallejaCL11}) and moreover (see
\cite{CallejaCL11b}) that Birkhoff invariants near a Lagrangian
invariant torus with a dynamics conjugated to a rotation disappear
(i.e., given a Lagrangian invariant torus with a dynamics conjugated
to a rotation, there is an analytic and symplectic change of variables
defined in a neighborhood of the torus, that conjugates the dynamics
to a rotation in the angles and a multiplication by a constant in the
actions).

The papers \cite{Locatelli, StefanelliL15} develop a Kolmogorov theory
based on transformation theory for quasi-periodic solutions in
quasi-integrable conformally symplectic systems and implement it
numerically. The paper \cite{CellettiC09} develops an a-posteriori KAM
theory for the spin-orbit problem (a two degrees of freedom
conformally symplectic system).

\subsection{Description of the set up}
\label{sec:setup}
We will consider analytic families of mappings or flows with a small parameter $\eps$
and having also internal parameters $\mu$. That is, given
an analytic symplectic manifold $\M$ of dimension $2d$ with exact symplectic form $\Omega$,
we will consider families of
mappings $f_{\mu,\eps}:\M\rightarrow\M$ satisfying:
\beq{familiesdefined}
f_{\mu, \eps}^* \Omega = \lambda(\eps) \Omega\ , \qquad \lambda(0) = 1\,
\eeq
or families of flows $\F_{\mu,\eps}$ such that:
\beq{fl}
L_{\F_{\mu, \eps}} \Omega = \chi(\eps) \Omega\ , \qquad \chi(0) = 0\ ,
\eeq
where $\lambda=\lambda(\eps)$, $\chi=\chi(\eps)$ are the conformal factors for maps and flows, respectively,
the star in \equ{familiesdefined} denotes
the pull-back, $L_{\F_{\mu, \eps}}$ in \equ {fl} is the Lie derivative.
We will refer to $\lambda(\eps)$, $\chi(\eps)$ as the dissipation.

Since we will discuss analyticity, all the parameters will be taken to be complex.
In applications the parameters are often real and then the values of the functions
are real. It will happen that all the calculations
we perform respect the properties that real arguments
of the function lead to real results.

In formula \eqref{familiesdefined}, $\eps$ is a small parameter that controls
the dissipation. The parameters $\mu\in\complex^d$ are some intrinsic parameters of the model
that are called \emph{the drift} in some papers.  Of course, the case when
$\lambda = 1$ (respectively, $\chi = 0$) corresponds to the mapping $f_{\mu,\eps}$ (respectively,
the flow $\F_{\mu,\eps}$) being symplectic.

Note that the interpretation of \eqref{familiesdefined} or \eqref{fl} is that the mapping or
the flow transforms the symplectic form into a multiple (could be a complex multiple for complex $\eps$) of
itself. We assume that the conformal factor is of the form
$$
\lambda(\eps) = 1 + \alpha \eps^a + O(|\eps|^{a+1})\ ,
$$
where $a\in\integer_+$, $\alpha\in\complex\backslash\{0\}$ (in many applications
$\alpha\in\real\backslash\{0\}$, so as to preserve the property that real arguments
lead to real variables).

We will fix $\omega$ within the set of Diophantine vectors (see
Definition~\ref{vectordiophantine} for the standard definition for
maps and Appendix~\ref{sec:flows} for the standard definition for
flows). We will be interested in studying the domain of complex values
of $\eps$ for which we can continue a KAM torus invariant for the
symplectic system.

Of course, we will assume several non-degeneracy conditions
(similar to the twist condition)  to prove the existence of KAM tori.
These non-degeneracy conditions can be verified
by performing some calculations on the approximate solution.

Note, that we will not assume that the Hamiltonian system is
integrable or nearly-integrable, but only that it has a KAM torus of
frequency $\omega$. In particular, the results apply to perturbations
of islands generated by resonances at higher values of the
perturbation (\cite{Duarte1994}, \cite{Duarte1999}).

For simplicity we will describe in detail the case of maps, while flows are discussed in
Appendix~\ref{sec:flows}.

The way we seek invariant tori of mappings is to try to find an
embedding $K_\eps: \torus^d \rightarrow \M$ and a parameter vector
$\mu_\eps\in\complex^d$ in such a way that
\begin{equation}\label{invariance0}
f_{\mu_\eps, \eps}\circ K_\eps = K_\eps \circ T_\omega\ ,
\end{equation}
where $T_\omega$ is the shift map defined  by $\omega$,
$T_\omega(\theta)=\theta+\omega$. We will fix $\omega$ to be Diophantine
(see Section~\ref{sec:diophantine}).
Of course, the equation
\eqref{invariance0} will have to be supplemented by some normalization
conditions, which ensure that the solutions are locally unique. We
refer to \cite{CallejaL09, CallejaL10b, LlaveR90, CallejaF11}
for a method to find invariant curves
that solve the invariance equation \equ{invariance0} numerically.

Our main result, Theorem~\ref{main}, shows that, if there is a
solution of \eqref{invariance0} for $\eps = 0$ (the symplectic case),
which satisfies some mild non-degeneracy conditions, we can find
$K_\eps$ and $\mu_\eps$ defined for a set $\G$ of $\eps$. The
functions $K_\eps$ and the vectors $\mu_\eps$ are analytic in  $\eps$
when $\eps$ ranges in the interior of the sets which
are described in Section~\ref{geometry}.  They also extend
continuously to the boundary of $\G$.
We anticipate that the sets
$\G$ do not include any ball centered at the origin in the complex
plane, even if they contain the origin in their closure.  On the other
hand, the sets $\G$ fail to include a ball centered at the origin by
very little.  As we will see, the domains are obtained by taking from
the ball centered at the origin
a sequence of smaller balls  centered along smooth lines going
through the origin and with radii decreasing faster than any power
of the distance of the center to the origin (see
Section~\ref{geometry} and Figure~\ref{fig:N0}).

Similar analyticity domains appeared in \cite{JorbaLZ00}, where the
authors used a strategy close to ours in considering the domains of
analyticity of resonant tori in near-integrable systems\footnote{The
problem considered in \cite{JorbaLZ00} is actually more singular
than the one considered here, since the twist and other
non-degeneracy constants depend on $\eps$ and they degenerate as
$\eps \to 0$. Nevertheless, they degenerate like a fixed power of
$\eps$ and one can construct approximations to any order.}  (see
also \cite{GentileG05,GGG2006,CostinGGG07} for results based on
resummation of series for the same problem as \cite{JorbaLZ00}).  The paper
\cite{JorbaLZ00} also obtained other geometric results such as the
monodromy of the stable and unstable bundles, which are not present in
other treatments.

We note that the domains we obtain here have several cuts
and that the analyticity domains do not contain sectors centered at
the origin with aperture bigger than $\pi/a$.
Hence, one cannot use only general complex analysis methods (\cite{Hardy, Balser, Sokal})
to deal with the asymptotic series, as these methods do
not guarantee that there is only a function with the same asymptotic expansion
(\cite{PhragmenL08,SaksZ65}). On the other hand, it is a byproduct of our analysis that the expansions
of solutions of the equations using the Lindstedt procedure are indeed unique.
In Section~\ref{sec:whitney} we also discuss $C^\infty$-Whitney properties of the solution.

The argument we present to prove our results has two ingredients:
\begin{itemize}
\item[$(i)$]
An \emph{a-posteriori} KAM theorem  (Theorem~\ref{mainKAM})
for conformally symplectic systems with complex parameters, which shows that if there is an approximate solution of
the invariance equation \eqref{invariance0}, then there is a true solution nearby.

Theorem~\ref{mainKAM} is a very quantitative statement on when a solution is
approximate enough to be the starting point of an iterative algorithm which converges.
The main condition of Theorem~\ref{mainKAM} is that the initial error is small enough. The precise
smallness condition depends mainly on the number theoretic properties of the complex number $\lambda$ -- the conformal
factor -- with respect to the frequency $\omega$. The smallness condition depends also
on the Diophantine properties of $\omega$ and on some non-degeneracy conditions
of the map. The Theorem~\ref{mainKAM} is a small modification of Theorem 20 in
\cite{CallejaCL11}.

\item[$(ii)$]
An algorithm to produce a perturbative series expansion that provides approximate solutions to
all orders in $\eps$, see part A of Theorem~\ref{main}.

This algorithm  is quite a standard procedure, which goes back at least to
\cite{Poincarefrench} and is based
on earlier literature.

In our case, taking advantage of the fact that the maps
are conformally symplectic (and, hence,  automatically  reducible) we can improve
the classical results in several directions:
In  Section~\ref{sec:quadratic} we present a quadratic algorithm to
compute the Lindstedt series, which is much faster and which could replace part A of Theorem~\ref{main}.
Indeed, the quadratic procedure in Section~\ref{sec:quadratic} gives  an alternative
proof of all of Theorem~\ref{main}.
We also note that using the automatic reducibility we can develop these series starting at any invariant torus. This leads to
improvements on the domain of analyticity that are discussed in Section~\ref{sec:improved}
as well as to Whitney regularity in the boundary.
\end{itemize}


Theorem~\ref{main} is proved by showing that an iterative procedure
(which is explicitly described in Algorithm~\ref{alg:step} and which is
a very practical numerical algorithm implemented in \cite{CallejaC10, CallejaF11})
converges if the initial error is small.
We point out that, if we start the iterative procedure not on one parameterization
and a drift, but on an analytic family of parameterizations and
drifts indexed by a variable $\eps$, then the result will also be an analytic family
in $\eps$
(the iterative procedure -- see Algorithm~\ref{alg:step} -- consists in applying algebraic operations,
taking derivatives and
solving cohomology equations with constant coefficients; all these elementary steps
preserve the analytic dependence on parameters).
We also point out that the convergence will be uniform in a domain of $\eps$, if all the non-degeneracy conditions and the initial error
are uniform.

Hence, if we start the iterative procedure
with a function analytic in an open set (respectively, continuous in a closed set)
of $\eps$ (such as that produced by part A in Theorem~\ref{main}), then
we obtain that the limit is analytic (respectively, continuous) in $\eps$ in the domain where the
convergence is uniform.

Since, furthermore, we have local uniqueness of the normalized solutions of the invariance equation
that satisfy a normalization condition,  we are sure that the
solutions obtained in two open sets agree on the overlap, and hence we can use arguments based on
analytic continuation.

For simplicity, we will present the proof for maps and we provide
the changes needed to obtain the result for flows in
Appendix~\ref{sec:flows}.\\

This paper is organized as follows.
In Section~\ref{sec:definitions}, we collect some of the standard definitions;
in Section~\ref{sec:diophantine} we provide some definitions on Diophantine properties, while
in Section~\ref{motivationgeometry}
we present the geometry of several sets where the conformal factors
satisfy Diophantine conditions.
In Section~\ref{sec:inveq} we present the invariance equation
and the normalization conditions we impose to obtain local uniqueness.

In Section~\ref{sec:main} we state Theorem~\ref{main},
the main result of this paper.

In Section~\ref{sec:KAM} we present Theorem~\ref{mainKAM}, which is
the \emph{a-posteriori} theorem alluded in $(i)$ above. Such a
theorem is the first ingredient
of the main result.  The proof of Theorem~\ref{mainKAM} is very
similar to the proof of Theorem 20 in \cite{CallejaCL11} and we just
outline the (rather minor) differences.  The existence of the
Lindstedt series is discussed in Section~\ref{sec:Lindstedt}.  In
contrast to the procedure using Lindstedt series (\cite{CallejaCL11}),
the present treatment allows one that the conformal factor depends on
$\eps$.  This provides the proof of the first part of
Theorem~\ref{main}, while the second part is proved in
Section~\ref{sec:partB}.  In Section~\ref{sec:geometricset} we
study some geometric properties of the set.
In Section~\ref{sec:optimal} we include a conjecture on the
optimality of the results described in this work. In Section~\ref{fastlindstedt} we present
several results: we show that the Lindstedt series expansions can be obtained around
any point; we present a relation with the theory of monogenic functions and establish
that the embedding function and the drift are
Whitney differentiable in the domain; we provide a
quadratic method for the computation of
the Lindstedt series, leading also to Part B of Theorem~\ref{main}; we present a
discussion on the improvement of the domain of analyticity.
The extension to the case of flows is provided in
Appendix~\ref{sec:flows}.

\begin{remark}
In many estimates in this paper, we can obtain that the domains
satisfy upper bounds less or equal than $C_N |\ep|^N$ for all $N$ and
for some constant $C_N$.  Clearly, since the bounds are true for all
$N$, one can get an upper bound less or equal than $\inf_N C_N |\ep|^N
\equiv \Gamma(\ep)$. The function $\Gamma$ will, of course, go to zero
faster than any power.

If one had explicit forms for $C_N$, it would be possible to obtain
explicit forms for $\Gamma$. In many problems similar to the ones we
are considering, one obtains \emph{factorial bounds} like $C_N = A
N^{aN}$ for some constants $a$ and $A$. In such a case, one obtains
that $\Gamma(\ep) = \exp(-C'|\eps|^{-1/a})$ for some constant $C'$. We
conjecture that the sizes of the balls to be excluded in the present paper satisfy the
factorial bounds.
\end{remark}

\section{Some definitions}\label{sec:definitions}

In this Section, we collect some definitions on
spaces, Diophantine properties and we set the notation.
Most of the definitions in this section are standard.
One non-standard definition that will play an important role is
the Diophantine property of complex numbers with respect to
a Diophantine frequency (see Definition~\ref{lambdadiophantine}).

Given $\rho > 0$ we define the complex extension of the $d$--dimensional torus as
\[
\torus^d_\rho = \{ z \in \complex^d/\integer^d \, :\ {\rm Re}(z_j)\in\torus\ ,\quad |\Im(z_j)| \leq \rho\ ,\quad j=1,...,d\}\ .
\]

We define $\A_\rho$ to be the vector space of functions analytic in Int($\torus^d_\rho$) and
which extend continuously to the boundary of $\torus^d_\rho$.

We endow $\A_\rho$ with the supremum norm
\begin{equation} \label{norm}
|| f ||_\rho = \sup_{\theta \in \torus^d_\rho} | f(\theta)|\ .
\end{equation}
The norm \eqref{norm} makes the space $\A_\rho$ into a Banach
space, indeed a Banach algebra under multiplication.
An important closed subspace of $\A_\rho$ is
the set of functions which take real values for real arguments.

Analogous definitions are made for  analytic functions on $\torus^d_\rho$ taking values
in vectors or in matrices and, of course, the multiplicative properties
of norms of  vectors and matrices lift to
supremum norms of functions taking values in vectors or matrices.

We also note that we can define analytic functions of $\eps \in
\complex$ taking values in $\A_\rho$.
Following standard definitions, we say that an $\A_\rho$-valued
function $K$ is analytic in an open domain $\D\subset\complex$ when, for any $\eps_0\in\D$, we can represent for
all $|\eps - \eps_0|$ sufficiently small, $K_\eps = \sum_{n=0}^\infty
K_n (\eps - \eps_0)^n$, where the convergence of the infinite sum
happens in $\A_\rho$.  Note that, with this definition, it is clear
that an $\A_\rho$-valued analytic function is also an $\A_{\rho'}$-valued
analytic function for $\rho' \le \rho$.

It is remarkable that there
are many other definitions that are apparently weaker than the
definition above, but which turn out to be equivalent (see
\cite[Chapter III]{HilleP48}). In our case, the $K$ will be analytic
functions from $\G$ to some $\A_\rho$. In principle, the domain $\G$
could depend on $\rho$, but we will not include it in the notation.

We note that we are only obtaining lower bounds of the domain of
analyticity. The method of proof leads to estimates
for different $\rho$'s, which are proved by selecting a different
constant. Hence, the statements that are valid for all Diophantine constants
in a range are also valid for all the values of $\rho$ in a range.

We recall the classical Cauchy inequalities for derivatives and for
Fourier coefficients (see, e.g., \cite{SaksZ65,Russmann75}).

\begin{proposition}\label{prop:Cauchy}
For any $ 0 < \delta < \rho$ and for any function
$K \in \A_\rho$, denoting by $D^j$ the $j$--th derivative, we have:
\beqano
|| D^j K ||_{\rho - \delta} &\le& C_j \delta^{-j}|| K||_\rho\ , \nonumber\\
| \widehat K_k | &\le& || K ||_{\rho}\ e^{- 2 \pi |k| \rho }\ ,
\eeqano
for some constants $C_j$, where $\widehat K_k$ denotes the $k$--th Fourier coefficient of $K$ and $|k|=|k_1|+...+|k_d|$.
\end{proposition}

\subsection{Diophantine properties}\label{sec:diophantine}
In this Section we collect some definitions concerning Diophantine properties, which
will be needed to bound the small divisors appearing in the solution of the invariance equation.
The only non standard definition is Definition~\ref{lambdadiophantine}.

\begin{definition}\label{vectordiophantine}
Let $\omega \in \real^d$, $\tau \in \real_+$.
We define the quantity $\nu(\omega;\tau)$ as
\begin{equation}\label{divisors_simple}
\nu(\omega;\tau) \equiv \sup_ {k \in \integer^{d} \setminus \{0\}}
|e^{2\pi i k \cdot\omega} -1 |^{-1} |k|^{-\tau}\ .
\end{equation}
In the sup above we allow $\infty$. Also, if
$|e^{2\pi i k \cdot\omega} -1 | = 0 $, we set
$\nu(\omega;\tau) = \infty. $

We say that  $\omega\in \real^d$ is Diophantine of
class $\tau$ and constant $\nu(\omega; \tau)$, whenever
\beq{nubis}
\nu(\omega; \tau) < \infty\ .
\eeq
We denote by $\D_d(\nu, \tau)$ the set of Diophantine vectors in $\real^d$ of class $\tau$ and
constant $\nu$.
\end{definition}

Of course, if $\omega$ is Diophantine of class $\tau$, it will be
Diophantine of class $\psi$ for all $\psi \ge \tau$.
For the purposes of this paper, the value of the constant $\nu$ is more important than the
exponent $\tau$, so we will consider one fixed exponent
in the main theorems.

\begin{definition} \label{lambdadiophantine}
Let $\lambda \in \complex$,  $\omega \in \real^d$, $\tau \in \real_+$.
We define the quantity $\nu(\lambda; \omega, \tau)$ as
\begin{equation}\label{divisors_lambda}
\nu(\lambda; \omega, \tau) \equiv
\sup_
{k \in \integer^{d} \setminus \{0\}}
|e^{2\pi i k \cdot\omega} -\lambda |^{-1} |k|^{-\tau}\ .
\end{equation}
Again, we allow $\infty$ in the supremum above
 and set $\nu(\lambda; \omega, \tau)=\infty$, if
$e^{2\pi i k \cdot\omega}=\lambda$ for some $k\in\integer^d\backslash\{0\}$.
\end{definition}

\begin{remark}\label{rem:semicontinuous}
We note that, for a fixed $\omega$, the function $\nu=\nu(\lambda;\omega,\tau)$ is a lower
semi-continuous function of $\lambda$, since it is the supremum of
continuous functions.
\end{remark}

\begin{remark}\label{rem:tau}
If $\nu(\lambda; \omega, \tau ) <  \infty$, then
for all $\psi \ge \tau$, we have  $\nu(\lambda; \omega, \psi) <  \infty$.

Note that the definition of $\nu(\lambda;\omega, \tau)$ does not require that
$\omega$ is Diophantine of exponent $\tau$.
\end{remark}

\begin{remark}\label{rem:lambda}
No matter what $\omega$ is, if $|\lambda| \ne 1$, then from the inequality
$$
|e^{2\pi i k \cdot\omega} -\lambda | |k|^{\tau}\geq |e^{2\pi i k \cdot\omega} -\lambda |
\geq \Big| |e^{2\pi i k \cdot\omega}| -|\lambda| \Big|=\Big|1-|\lambda| \Big|\ ,
$$
one obtains that
$$
\nu(\lambda;\omega, \tau)\leq |1-|\lambda||^{-1} < \infty
$$
for any $\tau$.
\end{remark}

As a consequence of the above remark, the only case that needs to be
studied in detail is when $|\lambda| = 1$, in which case it could be
that $\nu (\lambda; \omega, \tau) = \infty$ (e.g., if $\lambda =
e^{2\pi ik \cdot \omega} $ for some $k \in \integer^d \setminus \{0\}$
or if $\omega$ is a Liouville number).

For $\tau>d-1$ the set of real vectors which are not Diophantine of class
$\tau$ has zero Lebesgue measure in $\real^d$.  That is, the union
over all $\nu>0$ of the sets of Diophantine vectors of class $\tau$
satisfying \equ{nubis} with constant $\nu$, has full Lebesgue measure
in $\real^d$. For any $\omega \in \real^d$, if
$\tau > d-1$, the set of $\lambda$ for which $\nu(\lambda;\omega,
\tau) < \infty$ is of full Lebesgue measure on the unit circle (see
\cite{Schmidt80}). When we consider complex $\omega\in\complex^d$,
the set of vectors that are not Diophantine of class $\tau$ has zero
measure when $2\tau>d-1$. This shows that complex Diophantine is easier
than Diophantine over the reals.

We will consider fixed the analytic function $\lambda(\eps)$,
which gives the conformal factor as a function of the
perturbing parameter $\eps$. In particular, we will assume that $\lambda(\eps)$ is analytic in a
neighborhood of zero and that $\lambda(0) = 1$. Hence,
we assume that $\lambda=\lambda(\eps)$ satisfies:

\vskip.1in

{\bf H$\bf{\lambda}\qquad\qquad\qquad\qquad\qquad$} $\lambda(\eps)-1 = \alpha \eps^a + O(|\eps|^{a+1})$

\vskip.1in

\noindent
for some $a>0$ integer, $\alpha\in\complex\backslash\{0\}$.

This will be one of the assumptions of our main result stated in Theorem~\ref{main}. Note that, since we are
considering analytic functions, the alternative to the existence of $\alpha$ and $a$ in {\bf H$\bf{\lambda}$} is
that $\lambda(\eps)\equiv 1$, that is that the maps $f_{\mu, \eps}$
are symplectic.
Hence, given the analyticity assumption, the
only content of {\bf H$\lambda$} is that  the perturbations indeed change
the symplectic character (as well as setting the notation for $a,\alpha$,
which will play a role in the the quantitative statements).

We consider  the function $\lambda(\eps)$ as given in
{\bf H$\lambda$.}

We
define the set $\G=\G(A; \omega,\tau, N)$ for some $A>0$, $N\in\integer_+$,
$\omega \in \real^d$ as
\begin{equation}\label{goodset}
\G(A; \omega,\tau, N)  = \{ \eps \in \complex : \quad
\nu(\lambda(\eps); \omega, \tau) |\lambda(\eps) - 1|^{N+1} \le A \}\ .
\end{equation}

We also introduce the notation
\beq{Gr0}
\G_{r_0}(A;\omega,\tau,N)  = \G \cap \{\eps\in\complex:\  |\eps| \le r_0\}\ ;
\eeq
the set $\G_{r_0}$ will typically be used for sufficiently small $r_0$.

Notice that $\G(A; \omega,\tau, N)$ is
the preimage under the function $\lambda$ of the set
\begin{equation}\label{Lambdadefined}
\Lambda(A; \omega, \tau, N)=
\{ \lambda \in  \complex  : \quad
\nu(\lambda; \omega, \tau) |\lambda - 1|^{N+1} \le A \}\ .
\end{equation}

The motivation for introducing these sets
will be discussed in Section~\ref{motivationgeometry}.
We anticipate that the set $\G$ is the set
where the Diophantine constants of $\lambda(\eps)$ are not
too bad, so that a good approximation (up to a high power of
$\eps$) can be taken as the initial condition for an iterative
procedure that converges. More detailed motivations will be presented
in Section~\ref{motivationgeometry};
some depiction of the sets $\G_{r_0}$, $\A$ is presented in Figure~\ref{fig:N0}.

Since, for a fixed $\omega$, the function $\nu$ is lower semi-continuous, we
note that $\G$ is a closed set. The interior of this set is non-empty and
indeed, the set $\G$ is the closure of its interior.
We will be studying functions defined in $\G$ taking values
either in the complex or in some Banach space of
functions. We will consider functions on $\G$ which are
continuous in $\G$ and analytic in the interior of $\G$.
We will sometimes refer to these functions as analytic in $\G$.

For obvious typographical reasons, since  many of the arguments of
$\G, \Lambda$ will be fixed in the discussion, we will omit them:
if $\omega$, $\tau$, $N$, $\lambda$ are fixed in an argument, we will
just write $\G(A)$ or $\G_{r_0}(A)$. Of course, one could also add
$\rho$ to the dependence of $\G$, but as remarked above, this is not needed.

\subsection{Motivation for the role of the Diophantine constants.}
\label{motivationgeometry}

In this paper, we will fix $\omega$ Diophantine and search for
solutions of the functional equation \eqref{invariance0}, capturing
the geometric idea that we have a parameterization of an invariant torus.

The solutions of the invariance equation~\eqref{invariance0} are
obtained by an iterative method, whose step involves solving two
cohomology equations of the type considered in
Section~\ref{sec:cohomology} below (as well as algebraic and calculus operations).

One of the cohomology equations in the iterative step will involve
small divisors of the form $|e^{2\pi ik\cdot\omega}-1|^{-1}$ appearing
in \eqref{divisors_simple} and the other will involve small divisors
of the form $|e^{2\pi ik\cdot\omega}-\lambda|^{-1}$ appearing in
\eqref{divisors_lambda}.  Hence, the quantitative figure of merit of
the step will be the constant $\nu(\omega; \tau) \nu(\lambda; \omega,
\tilde \tau)$ for some $\tau,\tilde\tau\in\real_+$ (as well as other
factors, that we will consider fixed).

Since the two cohomological equations that we are going to solve in
Section~\ref{sec:KAM} are different, there is no reason
to impose that the exponents are equal, but, for the sake of
simplicity, we will not optimize these choices and, from now on, we
will just consider the case $\tilde\tau=\tau$.

Based on the results presented in \cite{CallejaCL11b}, we will show
that, if we start with an approximate solution of \equ{invariance0}
and if the initial error is small enough compared with the figure of
merit (the precise relation given in assumption {\bf H5} of
Theorem~\ref{mainKAM}), then the iterative procedure can be repeated
infinitely many times and it converges to a true solution. Taking as
initial condition the Lindstedt series, we will obtain convergence in
the sets claimed in Theorem~\ref{main}.

If we fix the frequency $\omega$ and the exponent $\tau$, the quality
factor for the step, namely $\nu(\omega; \tau) \nu(\lambda;
\omega,\tau)$, is a function of $\lambda$ alone. It will be important
to study the places where this function is not too big and identify
complex domains in the $\eps$-plane where it is
bounded uniformly. The values of $\lambda$
which lead to a quality factor in the iterative step, which is small enough with respect to the
initial error, are those for which the procedure converges.

As indicated in hypothesis {\bf H}$\lambda$, when we consider problems
depending on the complex parameter $\eps$, the quantity
$\lambda(\eps)-1$ will be approximately $\alpha \eps^a$. The
perturbation expansion to order $N$ will produce approximate solutions
that satisfy the equation to order $\eps^{N+1}$ (see \equ{formalpower}
in Theorem~\ref{main}). Hence, the sets $\G$ in \eqref{goodset} are
the sets where we can start the iterative procedure with an
approximate solution given by a perturbation expansion, and obtain
that the iteration process converges.  The sets $\G$ in
\eqref{goodset} will thus be the sets where we can establish that
there is a solution starting the iterative procedure from the approximate solution given by the
expansion. Of course, it is possible that the set of $\eps$ for which
$K_\eps$, $\mu_\eps$ are analytic
is larger than $\G$.  The set $\G$ is the place
where one particular method works, but
we could use other methods.  Indeed in Section~\ref{sec:improved}, we will find some way to
extend the domain $\G$.

We will undertake the study of the geometry of the sets $\G$ and $\Lambda$ in Section~\ref{geometry}
and in Section~\ref{sec:geometricset}.
We anticipate that, of course, one can obtain uniform bounds
for $\nu(\lambda;\omega, \tau)$ in any domain which is
away from the unit circle. The function $\nu(\lambda;\omega, \tau)$
is infinite in a dense set on the unit circle, but can be bounded
in domains that include the unit circle in the boundary,
indeed in a full measure set in the unit disk.
The geometry of these domains is rather interesting
and will be discussed in the next Section.

\subsubsection{Geometry of the level sets of  $\nu(\lambda;\omega,\tau)$ }
\label{geometry}

In this Section we try to understand the sets of points
described in Section~\ref{sec:diophantine}. More properties will be described in Section~\ref{sec:geometricset}.

The complement of the set $\Lambda$, say $\Lambda^c$, is the {\sl bad} set
where the bounds we use do not allow us to claim the
existence of tori, but we claim that if we take as initial condition of
the iterative procedure the Lindstedt polynomial of order $N$, then
the measure of $\Lambda^c$ is small for $N$ large enough.

The set $\Lambda^c$ is the union of the open  sets
\[
\begin{split}
\R_k &= \{ \lambda \in \complex \, :\
|e^{2 \pi i k\cdot \omega} - \lambda|^{-1}  >  A |k|^{\tau} | \lambda -1|^{-(N+1)}\}\\
&= \{ \lambda \in \complex \, :\
|e^{2 \pi i k\cdot \omega} - \lambda|  <  A^{-1} |k|^{-\tau}
| \lambda -1|^{N+1}\}\ .
\end{split}
\]

Since $\lambda$ appears on both sides of the inequality defining
$\R_k$, this is not easy to handle. On the other hand, we note
that if we consider the annulus $ \rho < |\lambda - 1| < 2 \rho$,
we see that the intersection of the set $\R_k$ with the annulus
 is contained in the ball
 $\B_k$
with center at $e^{2 \pi i k\cdot \omega}$ and radius
$C A^{-1} \rho^{N+1} |k|^{-\tau}$, for some constant $C>0$.

For each $k\in\integer^d\setminus \{0\}$, the area $\S_k$ of the ball is equal to
$$
\S_k=\pi\ C^2\ A^{-2}\ \rho^{2(N+1)}\ |k|^{-2\tau}\ .
$$
Notice that the area of the union of the balls $\B_k$ is finite, provided
$$
\sum_{k\in\integer^d\backslash\{0\}} |k|^{-2\tau}<\infty\ .
$$
Indeed, we see that as we consider $\rho$ small, the excluded
area decreases faster than $\rho^{2(N+1)}$, which is much smaller
than the area of the annulus, which is proportional to $\rho^2$.
Hence, we can say that $\Lambda$ has a point of density
at $1$ and that the excluded balls in the $\lambda$-space
decrease very fast as we approach $1$.
More detailed estimates will appear later.

The above remarks can be translated to the $\eps$ plane under the assumption
{\bf H$\lambda$}.

\begin{figure}[h]
\centering
\hglue0.2cm
\includegraphics[width=7truecm,height=10truecm]{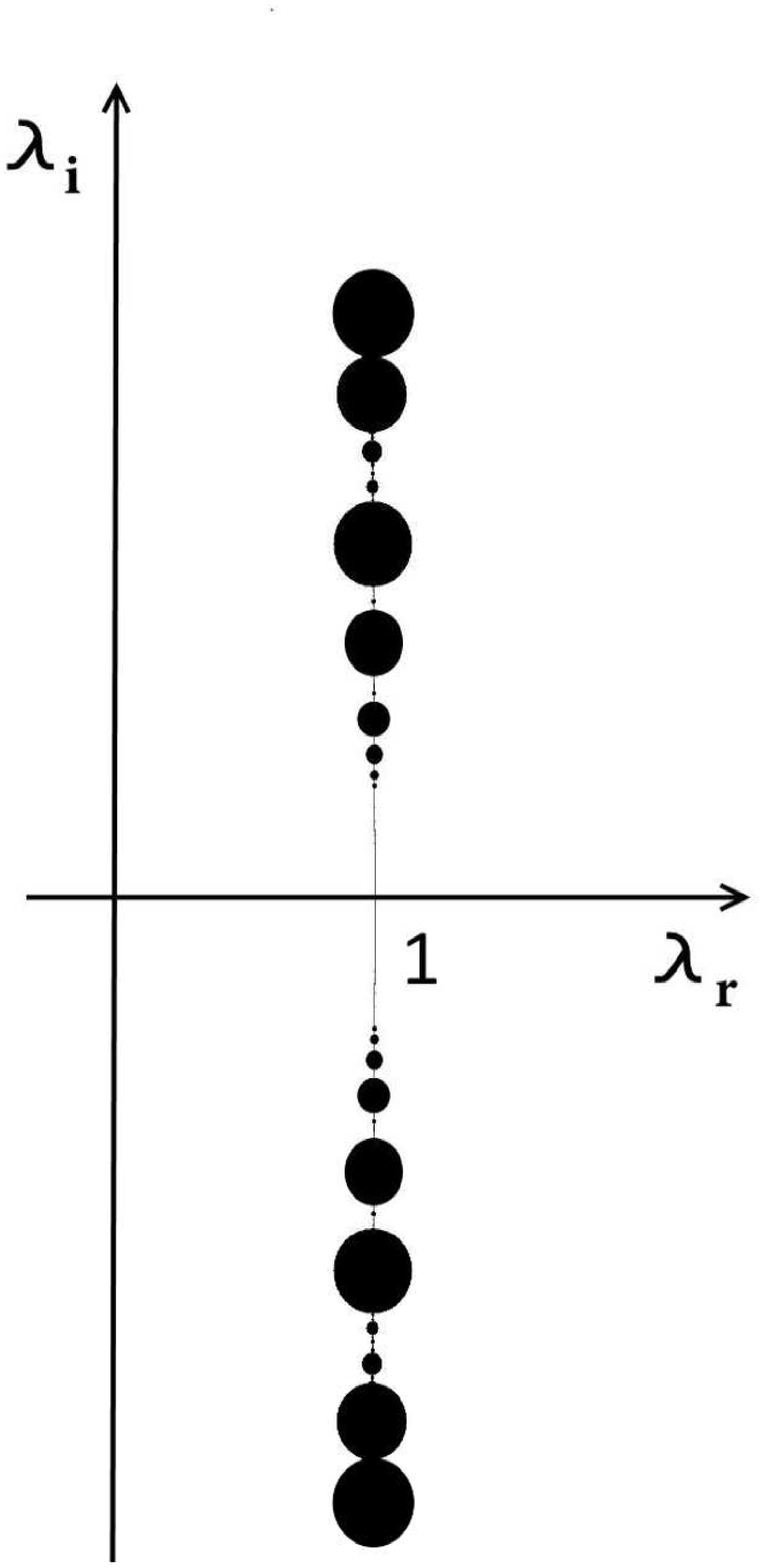}
\hglue1cm
\includegraphics[width=6truecm,height=10truecm]{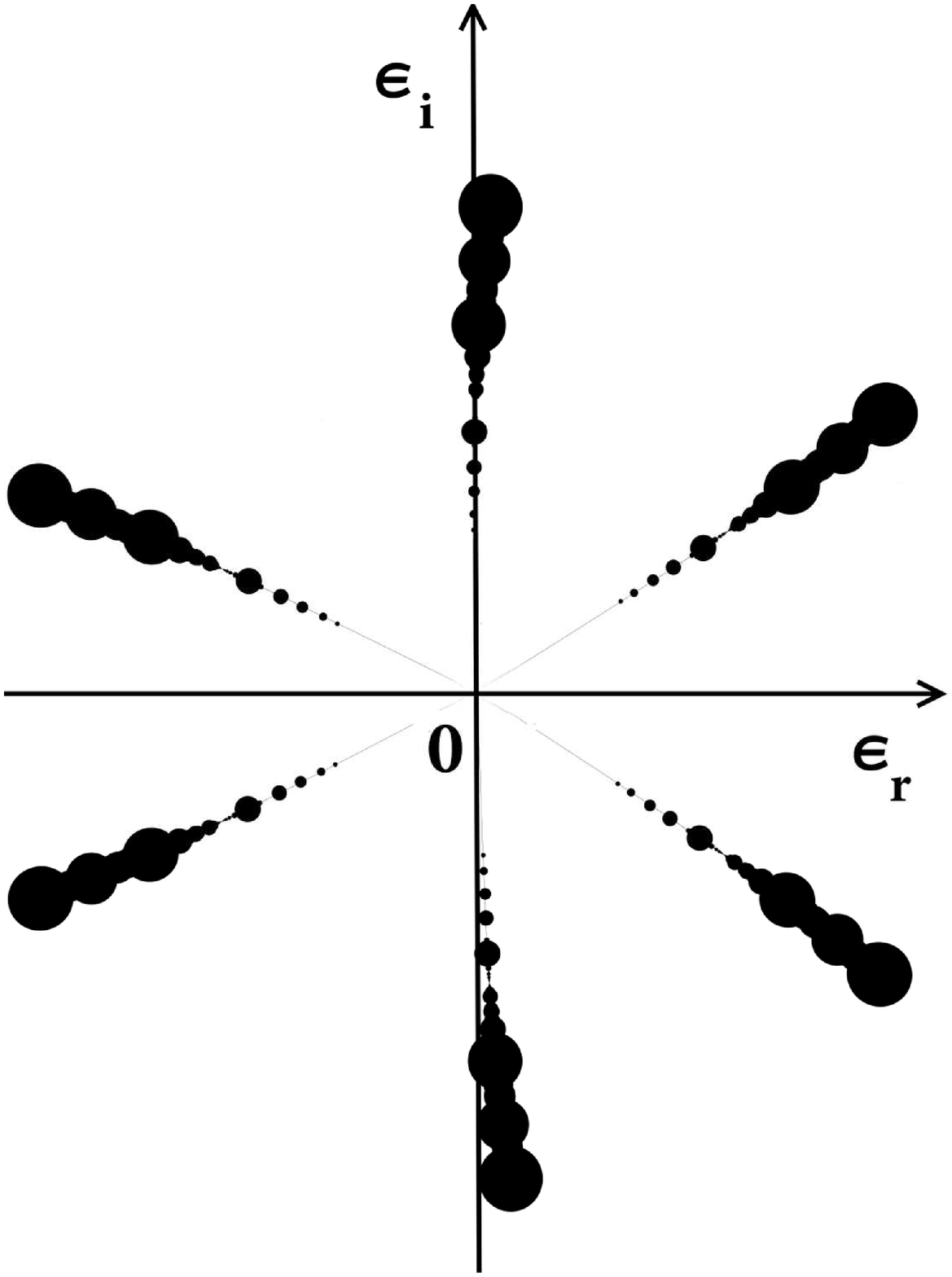}
\caption{A representation of the sets $\Lambda$, $\G$ introduced in
\equ{Lambdadefined}, \equ{goodset}; they are given by the region not
covered by the black circles. The domains are obtained for $d=1$,
$\tau=1$, $a=3$.  The radii of the balls have been rescaled for
graphical reasons.  Left: in black the complement in the complex
plane of the set $\Lambda$ as in \equ{Lambdadefined}; the centers of
the balls lie approximately on the unit circle, which at high
magnification resembles a vertical line.  Right: in black the
complement of the set $\G$ as in \equ{goodset}: the centers of
the balls are at
$|e^{2\pi ik\cdot\omega}-1|^{1\over a}$.}
\label{fig:N0}
\end{figure}

Figure~\ref{fig:N0} provides a representation of the circles which
must be excluded from the definition of the sets introduced in
\equ{goodset} and \equ{Lambdadefined}.  We remark that taking a higher
exponent $N$ in \equ{goodset}, \equ{Lambdadefined} does not alter
significantly the geometry of the domains, it just makes the radius
of the circles decrease faster. This fast convergence  to zero of
the radii
makes it impossible to represent the excluded regions  in a quantitatively
correct way. All the circles except for a few of them will be smaller
than a pixel!  Nevertheless, we conjecture that these
excluded circles are
there, see Section~\ref{sec:optimal}.

\subsubsection{Estimates of cohomology equations}
\label{sec:cohomology}
In this Section, we will present (rather elementary)
estimates on the solutions $\varphi:\torus^d\rightarrow\complex$ of twisted cohomology equations of the form:
\begin{equation} \label{twisted_cohomology}
\lambda \varphi(\theta) - \varphi(\theta + \omega) = \eta(\theta)\ ,
\end{equation}
where the function $\eta: \torus^d \rightarrow \complex$, the parameter
$\lambda \in \complex$ and the frequency vector $\omega \in \real^d$ are given.

We will assume that $\omega$ is Diophantine (see Definition~\ref{vectordiophantine})
and that $\lambda$ satisfies Definition~\ref{lambdadiophantine} with $\nu(\lambda;\omega,\tau)<\infty$.
We want to show that there exist solutions of
\eqref{twisted_cohomology} and that we obtain
quantitative estimates on the size of these solutions in
terms of the quantitative estimates of the Diophantine properties
of $\lambda$.
The estimates that we will obtain will be \emph{tame} estimates
in the sense of Nash-Moser implicit function theorems.

For subsequent applications in this paper, it will be quite important
that the estimates that we obtain about $\varphi$ are rather explicit
on $\nu(\lambda; \omega,\tau)$. Hence, they will hold uniformly in sets
of $\lambda$ for which $\nu(\lambda; \omega,\tau)$ is uniformly bounded. The
geometry of these sets was described explicitly in Section~\ref{geometry}.

Since we will be interested in the limit $\lambda \approx 1$,
it will be important that we can consider $\lambda$ ranging on
domains which include $1$ as a limit point. Hence, we will insert
in our hypotheses that $\omega$ is Diophantine to avoid
empty statements. The following Lemma is a standard result in KAM theory; the dependence
of the estimates on the exponents is not optimized as it is
done for $\lambda$ real according to \cite{Russmann75,Russmann76b}.
This will lead to slightly different estimates in {\bf H5} of Theorem~\ref{mainKAM}
when compared to those appearing in Theorem 20 of \cite{CallejaCL11}; a more detailed
comparison will be given later in Section~\ref{sec:KAM}.

\begin{lemma}\label{cohomology}
Let $\omega \in \D_d(\nu,\tau)$, $\lambda \in \complex$.
Assume that $\eta \in \A_\rho$, $\rho>0$, is such that $\int_{\torus^d}\eta(\theta) d\theta = 0$.
Then, we can find a unique $\varphi \in L^2(\torus^d)$ solving \eqref{twisted_cohomology},
that also satisfies
\[
\int_{\torus^d}\varphi(\theta) d\theta  = 0\ .
\]
Moreover, for any $0 < \delta < \rho$, we have $\varphi \in \A_{\rho -\delta}$, and
\begin{equation} \label{tame}
|| \varphi||_{\rho - \delta} \le C(\tau,d)\ \nu(\lambda; \omega,\tau)\ \delta^{-\tau-d}\
||\eta||_{\rho}
\end{equation}
for a suitable constant $C=C(\tau,d)$.
\end{lemma}

Note that a particular case of Lemma~\ref{cohomology} is
when $\lambda = 1$, which reduces to the most standard cohomology equations of
KAM theory.

\begin{proof}
The proof we present here is based on the most elementary (but possibly not optimal in the exponent for $\delta$) argument
(see Remark~\ref{rem:exponent} below).

We expand $\eta$ in Fourier series as $\eta(\theta) = \sum_{k \in \integer^{d} \setminus \{0\}}
\hat \eta_k e^{2 \pi i k \cdot \theta}$. We obtain that
\eqref{twisted_cohomology} is equivalent to having that for
all $k \in \integer^{d} \setminus \{0\}$:
\[
\lambda \hat \varphi_k - e^{2 \pi i k \cdot \omega} \hat \varphi_k = \hat \eta_k\ ,
\]
whose solution is
\[
 \hat \varphi_k = (\lambda - e^{2 \pi i k \cdot \omega})^{-1}\hat \eta_k\ .
\]
We then have the following estimate:
\[
\begin{split}
|| \varphi||_{\rho - \delta} & \le
\sum_{k \in \integer^{d} \setminus \{0\}} |\hat \varphi_k | e^{2\pi (\rho -\delta)|k|}
\leq \sum_{k \in \integer^{d} \setminus \{0\}}
 |\lambda - e^{2 \pi i k \cdot \omega}|^{-1} \, |\hat \eta_k| e^{2\pi (\rho -\delta)|k|} \\
&  \le
\sum_{k \in \integer^{d} \setminus \{0\}}
 \nu(\lambda;\omega, \tau) |k|^\tau ||\eta||_\rho\
e^{-2\pi \rho|k|}
e^{2\pi (\rho -\delta)|k|} \\
& =
\sum_{k \in \integer^{d} \setminus \{0\}}
 \nu(\lambda;\omega, \tau) |k|^\tau ||\eta||_\rho\
e^{-2\pi\delta |k|} \\
& \le
C  \nu(\lambda;\omega,\tau) ||\eta||_\rho
\sum_{j \in \nat}\ j^{\tau + d -1}
e^{-2\pi \delta j} \ ,
\end{split}
\]
where we have just used the Cauchy estimates for the Fourier coefficients
in terms of the supremum in a band and the definition of
the constant $\nu(\lambda;\omega,\tau)$
in \equ{divisors_lambda}. The desired result \equ{tame} just follows from estimating the last sum,
which is easily shown to be asymptotically bounded by $\delta^{-\tau-d}$.
\end{proof}

\begin{remark}\label{rem:exponent}
When $\lambda \in \real$, the papers \cite{Russmann75,Russmann76a}
contain more sophisticated estimates that lead to conclusions in
Lemma~\ref{cohomology} with a better exponent on $\delta$. Namely,
with the same notation of Lemma~\ref{cohomology}, when $\lambda
\in \real$, \cite{Russmann75,Russmann76a} lead to the conclusion
that $\| \varphi \|_{\rho -\delta}  \le C \nu \delta^{-\tau}
\|\eta \|_\rho$. Note that not only the exponent in
\cite{Russmann75} is better than the exponent in
Lemma~\ref{cohomology}, but also that the constant appearing in
the bounds above is proportional to the Diophantine constant of
$\omega$.

Unfortunately, when $\lambda \in \complex \setminus \real$, it seems
that it would be necessary to reexamine carefully the proof of
\cite{Russmann75,Russmann76b}. This would indeed be quite interesting on its own merit,
but will not be pursued here since, for
the purposes of this paper, the geometry of the level sets of the constants in the
estimates in Lemma~\ref{cohomology} is more
important than the exponents. The straightforward argument above leads
to constants $\nu$, whose geometry is easy to analyze.

Note that the Diophantine properties of $\omega$ enter
into the proof of Theorem 20 in \cite{CallejaCL11} only through the
estimates of the linearized equations. Hence, using the elementary
estimates in Lemma~\ref{cohomology} is equivalent -- for the proof of
Theorem 20 in \cite{CallejaCL11} -- to considering $\omega$ with a different
Diophantine exponent.
\end{remark}

\begin{remark}
We recall that, due to Remark~\ref{rem:lambda}, when $|\lambda|\not=1$, there are no small divisors
in the solution of the cohomology equation \equ{twisted_cohomology}. Hence, when $|\lambda|\not=1$, we have that
$\|\varphi\|_{\rho-\delta}\leq C(||\lambda|-1|)^{-1}\ \delta^{-\tau-d}\ \|\eta\|_\rho$.
Indeed, $\|\varphi\|_{\rho-\delta}\leq C(||\lambda|-1|)^{-1}\ \delta^{-\tau}\ \|\eta\|_\rho$.\\
\end{remark}

\section{The invariance equation and some normalizations}
\label{sec:inveq}

\subsection{The invariance equation}

The centerpiece of our treatment will be the invariance equation
\begin{equation}\label{invariance}
f_{\mu_\eps, \eps}\circ K_\eps = K_\eps \circ T_\omega\ .
\end{equation}
We think of \eqref{invariance} as an equation for the parameters $\mu_\eps$ and
for the embedding $K_\eps$ once $\eps$ is fixed. Of course, when $\eps$ varies,
we can think of $K,\mu$ as functions of $\eps$.

We will develop a Newton's method that can start from
any approximate solution. Using the geometric properties of
the system, the equations can be reduced to constant coefficients.
This has been the idea of the KAM theory based on parameterization,
started in \cite{Llave01c,LlaveGJV05} for symplectic systems; the extension of
the formalism for  conformally symplectic systems is
developed in \cite{CallejaCL11}.

\subsection{Some normalizations and uniqueness}
We note that the equation \eqref{invariance} never has
unique solutions.
If $(K,\mu)$ is  a solution so is, for any $\sigma \in
\torus^d$,  $(K^{(\sigma)},\mu)$ where
$K^{(\sigma)}(\theta) = K(\theta + \sigma)$.
Of course, $K^{(\sigma)}$ describes the same torus as $K$, the only difference
is the origin of the parameterization.
So, we can hope to get uniqueness only if we impose a normalization
that fixes the origin in the $\theta$ variables.
Later we will see that, indeed, this change of the origin is
the only source of non-uniqueness and that, once we fix it in
any reasonable way, we obtain that the solutions of \eqref{invariance}
and the normalization are locally unique. For us the normalization is
important because it allows for analytic continuation of solutions
produced by different methods.

The following normalization has been found to be geometrically natural
and easy to implement numerically.

Given a reference embedding $K_0$, which we choose once and for
all, we can form the matrix $M$ obtained juxtaposing the matrices $DK_0$ and
$J^{-1}\circ K_0 DK_0 N$, say
\beq{Mdef}
M = [DK_0\,|\, J^{-1}\circ K_0 DK_0 N]\ ,
\eeq
where
\beq{Ndef}
N=(DK_0^\top\ DK_0)^{-1}
\eeq
is a normalizing matrix.
As we will see below (see \cite{CallejaCL11} for more details), the matrix $M$ provides a frame of reference
near the image of $K_0$. Notice that $DK_0$ transforms vectors
that are in the tangent space to the image of $K_0$ and that, since
$K_0(\torus^d)$ will be almost Lagrangian, we have
that $J^{-1}\circ K_0 DK_0$ -- the symplectic conjugate -- will be almost perpendicular.

We will say that a torus with embedding $K_0$ is normalized, when
\begin{equation} \label{normalized}
\int_{\torus^d} \Big[M^{-1} (\theta)( K(\theta) - K_0(\theta))\Big]_1 \, d\theta = 0\ ,
\end{equation}
where the subscript $1$ denotes that we take the first $d$ rows.

The geometric interpretation is that $M$ defines a particularly interesting system of coordinates
near the torus $K_0$. The quantity $M^{-1} (\theta)( K(\theta) - K_0(\theta))$ expresses the displacement from
$K_0(\theta)$ to $K(\theta)$ in this system of coordinates and our condition is that,
in this system of coordinates, the $\theta$ component of the displacement has zero average.

With reference to the normalization \eqref{normalized}, we recall also the following result (Proposition 26
from \cite{CallejaCL11}), which shows that it can be imposed without any loss of generality for
solutions that are close.

\begin{proposition}
\label{prop:normalize}
Let $K_1, K_2$ be solutions of \eqref{invariance},
$\| K_1 - K_2\|_{C^1}$ be sufficiently small
(with respect to quantities depending only on $M$ -- computed out of
$K_1$ -- and $f$).
Then, there exists  $\sigma \in \real^n$, such that
$K_2^{(\sigma)}=K_2 \circ T_\sigma$ satisfies \eqref{normalized}.
Furthermore:
$$
|\sigma | \le C\, \| K_1 - K_2\|_{C^0}\ ,
$$
where the constant $C$ can be chosen to be as close to $1$ as desired by
assuming that $f_\mu$, $K_1$, $K_2$ are twice
differentiable, $DK_1^T DK_1$ is invertible
and $\| K_1 - K_2\|_{C^0}$ is sufficiently small.

The $\sigma$ thus chosen is locally unique.
\end{proposition}

Note that Proposition~\ref{prop:normalize} shows that if we have a family of solutions close to $K_0$, we can
modify this family of solutions by composing them with a small displacement, so that they
are normalized solutions.

\section{Statement of the main result, Theorem~\ref{main}}\label{sec:main}
In this Section, we state the main result, Theorem~\ref{main}, which - under suitable assumptions on the
mapping, the frequency and some non--degeneracy conditions - allows us to prove the existence of an
exact solution of the invariance equation, analytic in the set $\G_{r_0}$ introduced in
\equ{goodset}.

\begin{theorem}\label{main}
Let $\M\equiv\torus^d\times\B$ with $\B\subseteq\real^d$ an open, simply connected domain with smooth
boundary; $\M$ is endowed with an analytic symplectic form $\Omega$. Let us denote by $J=J(x)$ the matrix
representing $\Omega$ at $x$, so that for any vectors $u$, $v$, one has
$$
\Omega_x(u,v)=(u,J(x)v)\ .
$$

Let $\omega \in \real^d$ satisfy Definition~\ref{vectordiophantine} for Diophantine vectors.

Let $f_{\mu,\eps}$ with $\mu\in\Gamma$, $\Gamma\subseteq \complex^d$ open, $\eps\in\complex$, be a family of conformally symplectic mappings,
that satisfy \eqref{familiesdefined} with conformal factor as in {\bf H}$\lambda$.

Assume that for $\eps = 0$ the family of maps $f_{\mu,0}$ is symplectic.

Assume that for some value $\mu_0$ the map $f_{\mu_0, 0}$ admits a Lagrangian invariant torus,
namely we can find an analytic embedding from
$\torus^d \rightarrow \M$, say
$K_0  \in \A_\rho(\torus^d, \M)$, such that
\beq{K0}
f_{\mu_0,0} \circ K_0 = K_0 \circ T_\omega\ ,
\eeq
where $K_0 \in \A_\rho$ for some $\rho > 0$.
Moreover, assume that $K_0$ is Lagrangian, namely that
$$
DK_0^\top\ J\circ K_0\ DK_0=0\ .
$$
Assume furthermore that the torus $K_0$ satisfies the following hypothesis.

{\bf HND} Let the following non--degeneracy condition be satisfied:
\beqano
\det
\left(
\begin{array}{cc}
  {\overline S_0} & {\overline {S_0(B_{b0})^0}}+\overline{\widetilde A_{10}} \\
  0 & \overline{\widetilde A_{20}} \\
 \end{array}%
\right) \ne 0\ ,
\eeqano
where the $d \times d$ matrix $S_0$ is defined as
\beqano
S_0(\theta)&\equiv& N(\theta+\omega)^\top\ DK_0(\theta+\omega)\ Df_{\mu_0,0}\circ K_0(\theta)\ J^{-1}\circ K_0(\theta)\ DK_0(\theta) N(\theta)\nonumber\\
&-&N(\theta+\omega)^\top DK_0(\theta+\omega)^\top J^{-1}\circ K_0(\theta+\omega)\ DK_0(\theta+\omega)N(\theta+\omega)
\eeqano
with $N$ as in \equ{Ndef}, the $d \times d$ matrices
$\widetilde A_{10}$, $\widetilde A_{20}$ denote\footnote{We call attention that in
\cite{CallejaCL11} the statement of Theorem 20 defines the matrices ${\widetilde A}_1,$
${\widetilde A}_2$ as the first
$d$ and the last $d$ columns of the $2d \times d $ matrix $\widetilde A$.
Clearly, this sentence does not make sense, unless one changes
``columns" by rows. Since the sentence as written in \cite{CallejaCL11} is clearly impossible
and the detailed calculations are given, we hope that this has not
misled the readers, but we take the opportunity to set the record straight.
See also the discussion in Section~\ref{sec:KAM}.} the first $d$ and
the last  $d$ rows of the $2d\times d$ matrix
$\tilde A_0=(M\circ T_{\omega})^{-1}\ (D_{\mu} f_{\mu_0,0} \circ K_0)$,
where $M$ is as in \equ{Mdef}, $(B_{b0})^0$ is the solution (with  zero average)
of the cohomology equation $(B_{b0})^0-(B_{b0})^0\circ T_\omega=-(\widetilde A_{20})^0$,
where $(B_{b0})^0\equiv B_{b0}-\overline{B_{b0}}$ and the overline denotes the average.

Then, we have the following results.

\medskip

A) We can find a formal power series expansion
$K_\eps^{[\infty]}= \sum_{j=0}^\infty\eps^j  K_j$
satisfying \eqref{invariance} in the sense of formal power series.

More precisely, defining $K_{\eps}^{[\le N]} = \sum_{j=0}^{N}\eps^j  K_j$,
$\mu_{\eps}^{[\le N]} = \sum_{j=0}^{N}\eps^j  \mu_j$ for any $N\in\nat$ and $\rho>0$, we have
\begin{equation}\label{formalpower}
|| f_{\mu_{\eps}^{[\le N]}, \eps} \circ K^{[\le N]}_\eps -  K^{[\le N]}_\eps \circ T_\omega ||_{\rho'} \le C_N |\eps|^{N+1}
\end{equation}
for some $0<\rho'<\rho$ and $C_N>0$.

\medskip

B) We can find a set $\G_{r_0}$ of the form introduced in
\equ{Gr0} with $r_0$ sufficiently small and for any $0 < \rho' < \rho$, we can find the
functions
\[
\begin{split}
&K: \G_{r_0} \rightarrow \A_{\rho'}\ , \\
& \mu:  \G_{r_0} \rightarrow \complex^d\ ,
\end{split}
\]
which are analytic in the interior of
$\G_{r_0}$ and extend continuously to the boundary of $\G_{r_0}$,
such that for $\eps\in\G_{r_0}$ they satisfy exactly the invariance equation
\begin{equation}\label{isoslution}
f_{\mu_\eps, \eps} \circ K_\eps - K_\eps \circ T_\omega=0\ .
\end{equation}

Furthermore, we have that the solutions thus found have
the formal series provided in part A) as an asymptotic expansion. That is, for
any $N \in \nat$ and for any $0<\rho'<\rho$:
\beqa{Kmu}
||K^{[\le N]}_\eps -  K_\eps||_{\rho'}  &\le& C_N |\eps|^{N+1} \ ,\nonumber\\
|\mu^{[\le N]}_\eps -  \mu_\eps|  &\le& C_N |\eps|^{N+1}\ .
\eeqa
\end{theorem}

\begin{remark}
The condition {\bf HND} has a very transparent geometric interpretation, which we will present in the
proof given in Section~\ref{sec:Lindstedt}. See also \cite{CallejaCL11} for more details.

We also remark that the formal power series can be chosen to be normalized with respect to $K_0$.
The $K_\eps$ can also be chosen to be normalized.
\end{remark}

\section{Quantitative \emph{a-posteriori}  KAM theorem for
conformally symplectic systems}\label{sec:KAM}

The goal of this Section is to state a quantitative
KAM theorem, namely Theorem~\ref{mainKAM}, which is very similar to
Theorem 20 in \cite{CallejaCL11}. We will also detail the (rather
minimal) changes to the proof of Theorem 20 in \cite{CallejaCL11}, needed
to obtain a  proof of Theorem~\ref{mainKAM}.

Theorem 20  in \cite{CallejaCL11}  shows that, given a fixed
family of conformally
symplectic mappings and an approximate invariant torus for a value of
the parameter, we can find an exact invariant torus for a
nearby value of the parameter. The result presented in \cite{CallejaCL11}
is based on an a-posteriori format, which is very natural, because we do not
need to start from an integrable system.

The main novelty here with respect to Theorem 20 in \cite{CallejaCL11} is
that we highlight the dependence of the results on $\lambda$ -- the conformal
factor -- and
that we allow this conformal factor to be complex.

In \cite{CallejaCL11} the conformal factor $\lambda$ was considered
essentially fixed.  In some results of \cite{CallejaCL11}, $\lambda$
was allowed to range over a real interval $[1-A,1+A]$ for some $A>0$.
In this paper, however, $\lambda$ changes and the main goal of this
Section is to consider the dependence on $\lambda$.

The proof of Theorem~\ref{mainKAM} will be just walking through
the proof in \cite{CallejaCL11}, but keeping track of the
dependence of the constants on $\lambda$.
We remark, however, that the dependence in $\lambda$ comes
only through the Diophantine constant $\nu(\lambda;\omega,\tau)$
as in Definition~\ref{lambdadiophantine}.

The present treatment requires only minor
differences with that
of \cite{CallejaCL11}. More precisely,
\begin{itemize}
\item
In \cite{CallejaCL11} there is a treatment both of the analytic
and the finitely differentiable case.
In this paper, we
will only present the analytic case, since it is the one used in the applications of this paper.
\item
In \cite{CallejaCL11} there is a treatment both of the
case uniform in $\lambda$ and the case for fixed $\lambda$.
Moreover, in \cite{CallejaCL11} the parameter $\lambda$ was supposed to be real. In
our case, we will allow  $\lambda$ to be complex. We will obtain
estimates that, in the language of \cite{CallejaCL11},
are uniform in $\lambda$, but we will pay attention to the
Diophantine constants and the geometry of the sets in the
complex where these constants take values.
\item
In this paper we are using only crude bounds for cohomology equations,
whereas in \cite{CallejaCL11} we used the R\"ussmann estimates
\cite{Russmann75}. We do not know how to adapt the
R\"ussmann estimates to the case that $\lambda$ is complex (see Remark~\ref{rem:exponent}).
\end{itemize}

\begin{theorem}\label{mainKAM}
Let $\M\equiv\torus^d\times\B$ with $\B\subseteq\real^d$ an open, simply connected domain with smooth
boundary, endowed with a scalar product and a symplectic form $\Omega$.

Assume that the following hypotheses {\bf H1-H2-H3-H4-H5} are satisfied.

{\bf H1} Let $\omega\in\real^d$ be Diophantine of class
$\tau\in\real_+$ and constant $\nu(\omega; \tau)$.
For $\lambda\in\complex$ assume $\nu(\lambda; \omega, \tau)<\infty$, where $\nu(\lambda; \omega, \tau)$
is defined in \equ{divisors_lambda}.

{\bf H2} Let $f_{\mu,\eps}$ with $\mu\in\Gamma$, $\Gamma\subseteq\complex^d$ open, $\eps\in\complex$, be a family of (complex) conformally symplectic
mappings with respect to a symplectic form $\Omega$, that is
$f_{\mu,\eps}^* \Omega =  \lambda(\eps) \Omega\ $ (see \equ{familiesdefined})
with $\lambda(\eps)$ complex.

Let $K_a :\torus^d \to \M$,  $\mu_a \in \complex^d$, be an approximate solution of \equ{invariance0}, such that
\begin{equation}\label{initial_invariance}
f_{\mu_a,\eps}\circ K_a  - K_a \circ T_\omega  = E
\end{equation}
with error term $E$.

{\bf H3} Assume that the following non--degeneracy condition holds:
\begin{equation}
\label{non-degeneracy}
\det
\left(
\begin{array}{cc}
  {\overline S} & {\overline {S(B_b)^0}}+\overline{\widetilde A_1} \\
  (\lambda-1)\Id & \overline{\widetilde A_2} \\
 \end{array}%
\right) \ne 0\ ,
\end{equation}
where the $d \times d$ matrix $S$ is defined as
\beqa{Sdef}
S(\theta)&\equiv& N(\theta+\omega)^\top DK_a(\theta+\omega)\ Df_{\mu_a,\eps}\circ K_a(\theta)\ J^{-1}\circ K_a(\theta)\ DK_a(\theta) N(\theta)\nonumber\\
&-&\lambda\ N(\theta+\omega)^\top DK_a(\theta+\omega)^\top J^{-1}\circ K_a(\theta+\omega)\ DK_a(\theta+\omega)N(\theta+\omega)
\eeqa
with $N$ as in \equ{Ndef} with $K_a$ replacing $K_0$, the $d \times d$ matrices
$\widetilde A_1$, $\widetilde A_2$ denote the first $d$ and
the last  $d$ rows of the $2d\times d$ matrix
$\tilde A=(M\circ T_{\omega})^{-1}\ (D_{\mu} f_{\mu_a,\eps} \circ K_a)$,
where $M$ is as in \equ{Mdef} with $K_a$ replacing $K_0$, $(B_b)^0$ is the solution (with  zero average)
of the cohomology equation $\lambda (B_b)^0-(B_b)^0\circ T_\omega=-(\widetilde A_2)^0$,
where $(B_b)^0\equiv B_b-\overline{B_b}$ and the overline denotes the average.

We denote by ${\mathcal T}$ the quantity
$$
{\mathcal T}  \equiv \left \|\left(%
\begin{array}{cc}
  {\overline S} & {\overline {S(B_b)^0}}+\overline{\widetilde A_1} \\
  (\lambda-1)\Id & \overline{\widetilde A_2} \\
 \end{array}%
\right)^{-1} \right \|
$$
and we refer to ${\mathcal T}$ as the \emph{twist constant}.

Assume that $K_a\in \A_\rho$ for some $\rho>0$.
Assume furthermore
that for $\mu \in \Gamma$
we have that $f_{\mu,\eps}$ is a $C^1$--family of analytic functions on
a domain -- open connected set -- $\mathcal{C} \subset
\complex^d/\integer^d \times \complex^d $
with the following assumption on the domain.

{\bf H4} There exists $\zeta>0$, so that \beqano
&&{\rm dist}( \mu_a, \partial \Gamma) \ge  \zeta\ , \\
&&{\rm dist}(K_a(\torus^d), \partial\mathcal{C})  \ge
 \zeta\ .
\eeqano
Finally, let the solution be sufficiently approximate according to the following assumption.

{\bf H5} For some $0 < \delta < \rho$, the error term $E$ in \equ{initial_invariance} satisfies the inequality
$$
|| E||_\rho \le C\, \Big[\nu(\omega; \tau) \nu(\lambda; \omega, \tau)\Big]^{2} \, \delta^{4(\tau+d)}\ ,
$$
where $C$ denotes  a constant that can depend on $\tau$,
$d$, ${\mathcal T}$, $\| DK_a\|_\rho$, $\| N\|_\rho$, $\|M\|_\rho$
$\| M^{-1} \|_\rho$ as well as on $\zeta$ entering in
{\bf H4}.

Then, there exists  $\mu_\eps, K_\eps$ such that
$$
f_{\mu_\eps,\eps} \circ K_\eps - K_\eps \circ T_\omega = 0\ .
$$
The  quantities $K_\varepsilon$, $\mu_\varepsilon$ satisfy the
inequalities \beqano || K_\eps - K_a ||_{\rho-\delta} &\le& C_K\,
\nu(\omega;\tau)^{-1}\, \nu(\lambda;\omega,\tau)^{-1}\,
\delta^{-2(\tau+d)}\,  ||E||_\rho\ , \nonumber\\
| \mu_\eps - \mu_a| &\le& C_\mu\, ||E||_\rho\ ,
\eeqano
for positive constants $C_K$, $C_\mu$.
\end{theorem}

For a fixed value of $\eps$, say $\eps=\eps^*$, the solutions of \equ{invariance0} are locally unique, provided
the normalization condition \equ{normalized} is satisfied. Following \cite{CallejaCL11}, we have the following result.

\begin{lemma}\label{lem:unique}
Let $(K^{(1)},\mu^{(1)})$, $(K^{(2)},\mu^{(2)})$ be solutions of \equ{invariance0} with $\|K^{(1)}-K^{(2)}\|$ sufficiently small,
$\mu^{(1)}$, $\mu^{(2)}$ close enough. Assume that $K^{(2)}$ satisfies
\equ{normalized}, that the non--degeneracy condition {\bf H3} is satisfied at $(K^{(1)},\mu^{(1)})$ and that the assumption {\bf H4} on the
domain is satisfied.

Then, there exists $\sigma\in\real^d$, such that $K^{(2)}=K^{(1)}\circ T_\sigma$, $\mu^{(1)}=\mu^{(2)}$.
\end{lemma}

\subsection{Some notes on the proof of Theorem~\ref{mainKAM} and Lemma~\ref{lem:unique}}
\label{sec:theoremmainKAMproof}
Theorem~\ref{mainKAM} is proved in \cite{CallejaCL11},
but without keeping track of the difference
between $\nu(\omega;\tau)$ and $\nu(\lambda; \omega,\tau)$,
since only a fixed $\lambda$ was considered in \cite{CallejaCL11}.

For the sake of completeness, we will just repeat the main steps,
so that we can trace the constants and verify that the constants are
those we claimed in Theorem~\ref{mainKAM}. We omit details that can be
found in \cite{CallejaCL11}.

The procedure of \cite{CallejaCL11} is just to device
and estimate a Newton-like method based on some identities
obtained from the geometric properties of the
map. Using these geometric identities,
the equations appearing in Newton's method are reduced
to constant coefficient equations of the form appearing in
Lemma~\ref{cohomology}. This procedure is sometimes called \emph{automatic
reducibility}. Beside leading to a proof of the reducibility,
it leads to an efficient numerical algorithm
that was implemented in \cite{CallejaC10, CallejaF11}.
An important feature of the method, crucial
for our purposes, is that it can start from an approximate
solution, even if the map $f$ is far from integrable.

The Newton's method for equation \eqref{invariance} is based on the following steps. Given
an approximate solution $(K_a,\mu_a)$ of the invariance equation \equ{invariance0} as in \equ{initial_invariance}
with error term $E$,
then find corrections  $\Delta$ and $\sigma$ to $K_a$, $\mu_a$, respectively,
in such a way that:
\begin{equation} \label{Newton_correction}
\Big(D f_{\mu_a,\eps} \circ K_a\Big) \Delta - \Delta \circ T_\omega +
\Big(D_\mu f_{\mu_a,\eps} \circ K_a\Big) \sigma  = -E\ .
\end{equation}
Then, it can be proved that
$K_a + \Delta, \mu_a + \sigma$ will satisfy the invariance
equation with a much better accuracy (in a slightly smaller domain), precisely one expects that the new
error will be controlled by the square of the old one.

Unfortunately, the equation \eqref{Newton_correction} is
not so easy to deal with because it has non-constant coefficients.
The idea of the automatic reducibility method is to find
an adapted frame of coordinates that takes advantage of the geometry of the
system. Let us introduce $M$ like in \equ{Mdef} with $K_0$ replaced by $K_a$:
$$
M = [ DK_a\, |\,  J^{-1} \circ K_a DK_a N]
$$
with $J$ the symplectic matrix and $N$ the normalization factor
defined like in \equ{Ndef}:
$$
N = \left( DK_a^\top DK_a \right)^{-1}.
$$

The geometry of $M$ is that it is  a change of basis in the tangent
space of the approximately invariant torus.
The remarkable thing is that, in this basis, $Df_{\mu_a,\eps}\circ K_a$
has a very simple expression, namely
\begin{equation} \label{reduction}
Df_{\mu_a,\eps} \circ K_a\ M =  M \circ T_\omega
\begin{pmatrix}
&  \Id  &S \\
& 0  & \lambda \Id\\
\end{pmatrix}  +  R\ ,
\end{equation}
where $R$, measuring the error of the automatic reducibility for an approximately invariant torus,
is a quantity that  can be estimated by $E$ (in the  sense of Nash-Moser, we allow that the estimates
are \emph{tame} estimates in a smaller domain of analyticity).

The geometric reason for the identity \eqref{reduction}
is that, if we take derivatives of \eqref{initial_invariance},
we obtain the first $d$ columns of \eqref{reduction}.  Geometrically,
we have found a vector field $DK_a$ that gets mapped into itself by the
transformation $Df_{\mu_a,\eps}\circ K_a$. The directions $J^{-1}\circ K_a DK_a N $
are the symplectic conjugates to $DK_a$ (here one uses the fact that the
torus is approximately Lagrangian, which is established as a consequence
that $K_a, \mu_a$ satisfy \eqref{initial_invariance}).

Using \eqref{reduction} and writing the correction $\Delta$ for
$K_a$ as $\Delta = M W$ for some function $W$, we see that,
ignoring a term containing the factor $R\,W$, equation
\eqref{Newton_correction} becomes
\begin{equation}\label{quasinewton}
\begin{pmatrix}
&  \Id  &S \\
& 0  & \lambda \Id\\
\end{pmatrix}  W - W \circ T_\omega + \left( M\circ T_\omega\right)^{-1}D_\mu f_{\mu_a,\eps} \circ K_a \sigma =
-\left( M\circ T_\omega\right)^{-1} E\ .
\end{equation}

Note that \eqref{quasinewton} is an equation for both $W$ and $\sigma$.
Similar equations appear in KAM theory all the time. Note that
\eqref{quasinewton} becomes a cohomology equation of the form
\eqref{twisted_cohomology} for the second component $W_2$ of $W$ and, once this is solved,
we substitute $W_2$ in the equation for the first component $W_1$, which is an equation of
the form \eqref{twisted_cohomology} with $\lambda=1$.

Writing \equ{quasinewton} in components (i.e. taking the first $d$ rows and
the last $d$ rows) we obtain
\beqa{c12}
W_1-W_1\circ T_\omega&=&-S\, W_2-\widetilde E_1-\widetilde A_1\sigma\nonumber\\
\lambda W_2-W_2\circ T_\omega&=&-\widetilde E_2-\widetilde A_2\sigma\ ,
\eeqa
where $\widetilde E\equiv (M\circ T_\omega)^{-1}E$, $\widetilde A\equiv (M\circ T_\omega)^{-1} (D_\mu f_{\mu_a,\eps}\circ K_a)$.
We write $\widetilde A$ as $\widetilde A=[\widetilde A_1|\widetilde A_2]$, where $\widetilde A_1$, $\widetilde A_2$
denote the first $d$ and the last  $d$ rows of the $2d\times d$ matrix $\widetilde A$.

The solution of equations \equ{c12} requires that
the average of the right hand side is zero, but this can be accomplished by properly choosing the quantity
$\sigma$, provided that the non--degeneracy condition \equ{non-degeneracy} is satisfied.

This procedure is very standard in KAM theory, but in this case it has a complication.
If we change $\sigma$, since it is multiplied by a non-constant function, we change the solutions
of the equations which we have to seek for the average ${\overline W}_2$ and hence (compare with Step 9 in
Algorithm~\ref{alg:step} below) we change the solution for the average in the
equation for ${\overline W}_1$. Therefore, the equations for ${\overline W}_2$ and $\sigma$ are not completely decoupled. The observation
in \cite{CallejaCL11} is that we know that the dependence of $W_2$ on $\sigma$ is
affine and that we can compute the coefficients by solving cohomology equations
for the zero average part. If we do so and substitute in
the equation for ${\overline W}_2$, we are led to a linear equation for $\sigma$ when we impose
that the averages of both sides in \eqref{c12} match.

To do this computation explicitly, let us take the average of \equ{c12}, which leads to solving the following equations
for $\overline{W}_2$ and $\sigma$:
\beqa{ave1}
0&=&-{\overline S}\ {\overline W}_2-{\overline {S(B_a)^0}}-{\overline {S(B_b)^0}}\,\sigma-\overline{\widetilde E_1}
-\overline{\widetilde A_1}\sigma\nonumber\\
(\lambda-1)\ {\overline W}_2&=&-\overline{\widetilde E_2}-\overline{\widetilde A_2}\sigma\ ,
\eeqa
where $B_a$ and $B_b$ are such that $(B_a)^0$ solves the equation
$\lambda (B_a)^0 - (B_a)^0 \circ T_\omega = - (\widetilde E_2)^0$, while
$(B_b)^0$ solves the equation $\lambda (B_b)^0 - (B_b)^0 \circ T_\omega = - (\widetilde A_2)^0$.
Let us write ${\overline W}_2=A+B\sigma$ for some unknowns $A$, $B$.
Then, from the second of \equ{ave1} we obtain
$$
A=-{{\overline{\widetilde E_2}}\over {\lambda-1}}\ ,\qquad B=-{{\overline{\widetilde A_2}}\over {\lambda-1}}\ ,
$$
which, substituted in the first of \equ{ave1}, gives $\sigma$.

In summary, we are led to the following algorithm (which is identical to Algorithm 33 of \cite{CallejaCL11}).

\begin{algorithm}\label{alg:step}
Given $K_a: \torus^d \to \M$, $\mu_a \in \real^d$,
let $\lambda\in \complex$ be the conformal factor for the mapping $f_{\mu,\eps}$. Perform the following computations:
\begin{itemize}
\item[1)] $E\gets f_{\mu_a,\eps} \circ K_a - K_a\circ T_\omega$
\item[2)] $\alpha \gets DK_a$
\item[3)] $N\gets [\alpha^\top \alpha]^{-1}$
\item[4)] $M\gets [\alpha\,|\, J^{-1}\circ K_a\alpha N]$
\item[5)] $\beta\gets (M\circ T_\omega)^{-1}$
\item[6)] $\widetilde E\gets \beta E$
\item[7)] $P\gets\alpha N$
\item[{}] $\gamma\gets \alpha^\top J^{-1}\circ K_a\alpha$
\item[{}] $S\gets (P\circ T_\omega)^\top Df_{\mu_a,\eps} \circ K_a\ J^{-1}\circ K_a P-\lambda (N\circ T_\omega)^\top(\gamma\circ T_\omega) (N\circ T_\omega)$
\item[{}] $\widetilde A\gets (M\circ T_\omega)^{-1}\ D_\mu f_{\mu_a,\eps} \circ K_a$
\item[8)] $(B_a)^0$ solves \quad
$\lambda (B_a)^0 - (B_a)^0 \circ T_\omega = - (\widetilde E_2)^0$,
\item[{}] $(B_b)^0$ solves \quad
$\lambda (B_b)^0 - (B_b)^0 \circ T_\omega = - (\widetilde A_2)^0$
\item[9)] Find ${\overline W}_2$, $\sigma$ solving
\beqano
0 &=& - \overline{S}\, \overline{W_2}
- \overline{ S (B_a)^0   }
- \overline{ S (B_b)^0   } \sigma
 - \overline{\widetilde E_1}
 - \overline{\widetilde A_1}\sigma\\
(\lambda -1) \overline{W_2}
&=& -\overline{\widetilde E_2} - \overline{\widetilde A_2 }\sigma
\eeqano
\item[10)] $(W_2)^0=(B_a)^0+\sigma (B_b)^0$
\item[11)]  $W_2=(W_2)^0+{\overline W}_2$
\item[12)] $(W_1)^0$ solves $\noaverage{W_1} -
\noaverage{W_1}\circ T_\omega = - \noaverage{S W_2}
 -\noaverage { \widetilde E_1} - \noaverage{\widetilde A_1} \sigma$
\item[13)] $K_a \gets K_a + MW$
\item[{}] $\mu_a \gets \mu_a +\sigma$\ .
\end{itemize}
\end{algorithm}

\begin{remark}\label{worksforcomplex}
\begin{itemize}
\item[$(i)$] It is important to note that Algorithm~\ref{alg:step}
involves only algebraic operations, compositions of derivatives and solving
cohomology equations
which work just as well when some of
the objects involved are complex. Indeed, in \cite{CallejaCL11} -- and
in good part of KAM theory -- many functions are defined in complex
extensions of the tori.
\item[$(ii)$] We note that, besides being the basis of the theoretical treatment in
\cite{CallejaCL11}, this is also a very practical algorithm, since each of the steps are
obtained applying standard algebraic manipulations, common in Celestial Mechanics.
Notice that one step achieves quadratic convergence, but it is required a low
storage and low number of operations.
\end{itemize}
\end{remark}

A consequence of Remark~\ref{worksforcomplex} is the following result.

\begin{cor}\label{isanalytic}
Consider that we give as input to the Algorithm~\ref{alg:step}
the members
$K_\eps, \mu_\eps$ of a family indexed by $\eps$, with $\eps$ ranging in a domain. Assume that the
non-degeneracy assumption holds in the domain.

If $K_\eps, \mu_\eps$ is an analytic (respectively continuous) family,
then the result of the algorithm is also an analytic (respectively continuous)
family.
\end{cor}

To obtain estimates for the iterative step described in Algorithm~\ref{alg:step}, we observe that the bounds for
the correction $(W,\sigma)$ remain very similar to the estimates in
\cite{CallejaCL11}. The main difference with the procedure in \cite{CallejaCL11} is that in Step 8, we
use the estimates of Lemma~\ref{cohomology}.

The estimates for the error in the step do not need any change from the treatment in \cite{CallejaCL11}, since they
just involve adding, subtracting, using the second order Taylor estimates
and estimating the neglected term involving $RW$.

After the estimates for the step are performed, we only need to check that the
iteration can proceed and yields the desired results. This is nowadays
quite standard and does not require any changes from the presentation in \cite{CallejaCL11} to which we refer the reader
for full details.

We conclude by mentioning that the proof of Lemma~\ref{lem:unique} about the uniqueness of the solution does not
require any change from Theorem 29 in \cite{CallejaCL11}.

\section{Proof of A) in Theorem~\ref{main}}\label{sec:Lindstedt}
In this Section, we prove part A of Theorem~\ref{main}, which amounts to showing
the existence of Lindstedt series to all orders. Let us start from the exact solution $(K_0,\mu_0)$ as in
\equ{K0}; since we assumed that $f_{\mu_0,0}$ is symplectic, we have that
$$
f_{\mu_0,0}^*\Omega=\Omega\ .
$$
Let $M_0=[DK_0\, |\, J^{-1}\circ K_0\, DK_0\, N]$, $N=(DK_0^\top DK_0)^{-1}$ and let
$S_0=S$ with $S$ as in \equ{Sdef}; then, one obtains
\beq{f0}
Df_{\mu_0,0}\circ K_0(\theta)\ M_0(\theta)=M_0(\theta+\omega)\
\left(
\begin{array}{cc}
  \Id & S_0(\theta) \\
  0& \Id \\
 \end{array}%
\right)\ .
\eeq
Let $K_\eps^{[\leq N]}=\sum_{j=0}^N\eps^j K_j$, $\mu_\eps^{[\leq N]}=\sum_{j=0}^N\eps^j \mu_j$;
inserting these power series expansions in the invariance equation \equ{invariance}, expanding
the series in $\eps$ and equating the coefficients of the same power of $\eps$,
we obtain recursive relations defining $K_j$ and $\mu_j$, as described below.

At the first order in $\eps$ we obtain the equations:
\beq{eps0}
(Df_{\mu_0,0}\circ K_0)K_1-K_1\circ T_\omega +(D_\mu f_{\mu_0,0}\circ K_0)\mu_1+D_\eps f_{\mu_0,0}\circ K_0=0\ ,
\eeq
while at the $j$--th order in $\eps$, $2\leq j\leq N$, we get:
\beq{epsj}
(Df_{\mu_0,0}\circ K_0)K_j-K_j\circ T_\omega +(D_\mu f_{\mu_0,0}\circ K_0)\mu_j=F_j(K_0,...,K_{j-1},\mu_0,...,\mu_{j-1})\ ,
\eeq
where $F_j$ is an explicit polynomial in its arguments with coefficients depending on the derivatives of $f_{\mu_\eps,\eps}$
computed at $\mu_\eps=0$, $\eps=0$ and composed with $K_0$. We note for future reference that
$$
F_j(K_0,...,K_{j-1},\mu_0,...,\mu_{j-1})=-{1\over {j!}}{{d^j}\over {d\eps^j}} f_{\mu_0,0}\circ K_0+
\tilde F_j(K_0,...,K_{j-1},\mu_0,...,\mu_{j-1})\ ,
$$
where $\tilde F_j$ does not depend on derivatives of order $j$ of $f_{\mu_\eps,\eps}$.
Equation \equ{eps0} is of the same kind of \equ{epsj}, since it suffices to define
$F_1(K_0,\mu_0)\equiv -{d \over d \eps} f_{\mu_0,0}\circ K_0$.

To solve \equ{epsj} (equivalently \equ{eps0}), we write $K_j(\theta)\equiv M_0(\theta)W_j(\theta)$
for a suitable function $W_j=W_j(\theta)$, so that \equ{epsj} becomes
$$
(Df_{\mu_0,0}\circ K_0)\ M_0 W_j-(M_0 W_j)\circ T_\omega+(D_\mu f_{\mu_0,0}\circ K_0)\mu_j=F_j(K_0,...,K_{j-1},\mu_0,...,\mu_{j-1})\ .
$$
Using \equ{f0} we obtain:
\beqa{star}
\left(
\begin{array}{cc}
  \Id & S_0(\theta) \\
  0& \Id \\
 \end{array}%
\right)W_j-W_j\circ T_\omega&+&(M_0\circ T_\omega)^{-1}\ (D_\mu f_{\mu_0,0}\circ K_0)\mu_j\nonumber\\
&=&(M_0\circ T_\omega)^{-1}\ F_j(K_0,...,K_{j-1},\mu_0,...,\mu_{j-1})\ .
\eeqa
Writing \equ{star} in components, again we recall that this means
taking the first $d$ rows and the last $d$ rows,
say $W_j=(W_{j1} | W_{j2})$, we get the following equations:
\beqa{coh}
W_{j2}-W_{j2}\circ T_\omega+[(M_0\circ T_\omega)^{-1}(D_\mu f_{\mu_0,0}\circ K_0)]_2\ \mu_j&=&[(M_0\circ T_\omega)^{-1} F_j]_2\nonumber\\
W_{j1}-W_{j1}\circ T_\omega+S_0W_{j2}+[(M_0\circ T_\omega)^{-1}(D_\mu f_{\mu_0,0}\circ K_0)]_1\ \mu_j&=&[(M_0\circ T_\omega)^{-1} F_j]_1\ ,
\eeqa
where $[\cdot]_1$, $[\cdot]_2$ denote the first and second component. Under the non--degeneracy assumption
{\bf HND} and provided $\omega\in\D_d(\nu,\tau)$, equations \equ{coh} can be solved to determine $W_{j1}$, $W_{j2}$, $\mu_j$,
according to the following procedure. Let
$$
\widetilde E_j\equiv (M_0\circ T_\omega)^{-1}F_j\ ,\qquad \widetilde A_0\equiv (M_0\circ T_\omega)^{-1}(D_\mu f_{\mu_0,0}\circ K_0)\ ;
$$
then, \equ{coh} becomes:
\beqa{cohj}
W_{j2}-W_{j2}\circ T_\omega+\widetilde A_{20}\mu_j&=&\widetilde E_{j2}\nonumber\\
W_{j1}-W_{j1}\circ T_\omega+\widetilde A_{10}\mu_j&=&\widetilde E_{j1}-S_0W_{j2}\ .
\eeqa
Taking the average of the first equation in \equ{cohj}, we obtain
\beq{S1}
\overline{\widetilde A_{20}}\ \mu_j=\overline{\widetilde E_{j2}}\ ,
\eeq
which determines $\mu_j$. Taking the average of the second equation, we obtain
\beq{A}
\overline{\widetilde A_{10}}\mu_j=\overline{\widetilde E_{j1}}-\overline{(S_0 W_{j2})}\ .
\eeq
Let $W_{j2}=\overline{W_{j2}}+(W_{j2})^0$; then, we have
\beq{S0W}
\overline{(S_0 W_{j2})}=\overline{S_0}\ \overline{W_{j2}}+\overline{S_0(W_{j2})^0}\ .
\eeq
Using that $W_{j2}$ is an affine function of $\mu_j$, we write
$(W_{j2})^0\equiv (B_{a0})^0+ (B_{b0})^0 \mu_j$, where $(B_{a0})^0$, $(B_{b0})^0$ are the solutions of the equations:
\beqa{B}
(B_{a0})^0-(B_{a0})^0\circ T_\omega&=&(\widetilde E_{j2})^0\nonumber\\
(B_{b0})^0-(B_{b0})^0\circ T_\omega&=&-(\widetilde A_{20})^0\ .
\eeqa
From \equ{A}, \equ{S0W} and \equ{B} we obtain:
\beq{S2}
\overline{\widetilde A_{10}}\mu_j+\overline{S_0}\ \overline{W_{j2}}+\overline{S_0 (B_{b0})^0}\ \mu_j =
-\overline{S_0 (B_{a0})^0}+\overline{\widetilde E_{j1}}\ .
\eeq
The equations \equ{S1} and \equ{S2} in the unknowns $\overline{W_{j2}}$ and $\mu_j$ can be written as
$$
\left(
\begin{array}{cc}
  \overline{S_0} & \overline{\widetilde A_{10}}+\overline{S_0 (B_{b0})^0} \\
  0& \overline{\widetilde A_{20}} \\
 \end{array}%
\right)
\ \left(\begin{array}{c}
          \overline{W_{j2}} \\
          \mu_j \\
        \end{array}
\right)=\left(\begin{array}{c}
          -\overline{S_0(B_{a0})^0}+\overline{\widetilde E_{j1}} \\
          \overline{\widetilde E_{j2}} \\
        \end{array}
\right)\ ,
$$
which can be solved to obtain $\overline{W_{j2}}$, $\mu_j$,
provided {\bf HND} is satisfied.  The proof is completed, once we
solve the equations \equ{cohj} for the non--average parts of $W_1$ and
$W_2$.  The equations \equ{cohj} are cohomological equations of the
form \equ{twisted_cohomology} with $\lambda=1$.

Assuming that we  have determined the functions $K_j$ and the terms $\mu_j$
for $1\leq j\leq N$, we obtain the finite sums
$K_\eps^{[\leq N]}$, $\mu_\eps^{[\leq N]}$, which solve the invariance
equation within an error given in \equ{formalpower}.

\section{Proof of B) in Theorem~\ref{main}}\label{sec:partB}

We start by considering an approximate solution, as provided by part A of Theorem~\ref{main}, namely
a solution $(K_\eps^{[\leq N]}, \mu_\eps^{[\leq N]})$, which can be expanded in formal power series in $\varepsilon$ as
$$
K_\eps^{[\leq N]}=\sum_{j=0}^N \eps^j K_j\ ,\qquad \mu_\eps^{[\leq N]}=\sum_{j=0}^N \eps^j \mu_j\ ,
$$
and which satisfies the bound \equ{formalpower}.

Let $A>0$ and let $\eps_0\in\G_{r_0}(A)$ with the
 set $\G_{r_0}$ as in \equ{Gr0}, where the
cohomological equations can be solved. Assume that $\eps$ belongs to a sufficiently small ball
$\B$ centered in $\eps_0$, for example
\beq{eps}
|\eps-\eps_0|\leq 10^{-6}\ dist(\eps_0,\partial \G_{r_0}(A))\ .
\eeq
By the choice of $\eps$ in this subset of $\G_{r_0}(A)$, all the assumptions stated in Theorem~\ref{mainKAM} are satisfied.
In particular, the Diophantine condition {\bf H1} is required also in Theorem~\ref{main}. The approximate
solution in {\bf H2} is provided by the choice $(K_a,\mu_a)=(K_\eps^{[\leq N]}, \mu_\eps^{[\leq N]})$; the associated error term
can be assumed to be sufficiently small as in {\bf H5} thanks to the inequality \equ{formalpower}
and the assumption \equ{eps}, provided $r_0$ is sufficiently small.
We recall that due to {\bf H$\lambda$} we have $\lambda(\eps)-1=\alpha\eps^a+O(|\eps|^{a+1})$
for some $a\in\integer_+$, $\alpha\in\complex$, so that
the right hand side of \equ{formalpower} can be estimated by a power of $|\lambda(\eps)-1|$ as
$$
|| f_{\mu_\eps^{[\le N]}, \eps} \circ K^{[\le N]}_\eps -  K^{[\le N]}_\eps \circ T_\omega ||_{\rho'} \le C_N' |\lambda-1|^{{N+1}\over a}
$$
for some constant $C_N'>0$. By taking $\lambda$ close to 1, we obtain the bound in {\bf H5} on the error term.

Using the fact that the determinant is a continuous function,
the non-degeneracy assumption {\bf H3} is implied
by {\bf HND}, provided {\bf H}$\lambda$ is satisfied and $r_0$ is sufficiently small.

In conclusion, all assumptions required in Theorem~\ref{mainKAM} are satisfied and Theorem~\ref{mainKAM} allows one
to state that there exists an exact solution of the invariance equation \equ{isoslution}.

The conditions of Theorem~\ref{mainKAM} are verified uniformly and
therefore the sequence of approximate solutions constructed in the
proof of Theorem~\ref{mainKAM} converges uniformly to the true
solution $(K_\eps,\mu_\eps)$ satisfying \equ{isoslution}. Such families
of functions
will be analytic  in $\eps$ in the interior
of $\G_{r_0}$ and continuous in all of $\G_{r_0}$  as a consequence of Corollary~\ref{isanalytic}.

Due to the construction of the exact solution, the inequalities \equ{Kmu} are satisfied.

We conclude by mentioning that the error is as small as we want for $\eps$ small and $N$ large enough, and that the bounds
on the constants depend on the definition of the set \equ{eps}.
As remarked before, we can ensure that the solution satisfies the normalization \equ{normalized};
such normalized solutions are locally unique.

\section{Further geometric properties of the sets  $\G$, $\Lambda$}\label{sec:geometricset}

In this Section we formulate two geometric properties of
the set $\Lambda$ defined in
\eqref{Lambdadefined}.  From this, one can obtain properties of
the set $\G$  defined in \eqref{goodset}, since it is obtained from $\Lambda$
by a conformal transformation.

As a matter of fact, we will also consider the set
$$
\Upsilon(A; \omega, \tau)=
\{ \lambda \in  \complex  : \quad
\nu(\lambda; \omega, \tau) \le A \}\ ,
$$
which is easier to study. Clearly we have that in a neighborhood of
the unit circle
\[
\Lambda(A; \omega, \tau, N) \supset \Upsilon(2.1^{-(N+1)} A, \omega, \tau)\ .
\]
This inclusion is a not very accurate estimate
near $\eps = 1$, since we are ignoring the
factor $|\lambda - 1|^{N+1}$. On the other hand, for properties
over the whole unit circle, this is a good estimate.

We observe that, since the number $\nu$ is the supremum of
several quantities, the sublevel sets are
obtained by removing the sets where one of the inequalities required
in the definition of $\Upsilon$ fails.
We observe that the places where one of the inequalities fails can be
bounded from above by a ball and we can obtain a lower bound for
the set $\Upsilon$ by removing from the plane balls which enclose
the region where one of the inequalities fails.
These inequalities, and hence the excluded balls, are indexed by $k\in\integer^d\backslash\{0\}$.
The radii of the excluded balls decrease as $|k|$ grows.

Of course, the set  $\Upsilon $
contains open balls outside the unit circle. Hence
the interesting question consists in studying the properties of density of $\Upsilon $ on the unit circle.

In this Section we will establish  two geometric properties of the set
$\Lambda$: the first one states that one
is a point of density for the set $\Lambda$, both as
subset of the complex and also restricted to the unit circle (see
Proposition~\ref{density}); this result is independent of $A$.
The second property (see Proposition~\ref{pro:tang}) shows
that the set of points that are \emph{tangentially accessible}
(see Definition~\ref{tangentiallyaccessible} and compare with \cite{Caratheodory1},
\cite{Caratheodory2}, \cite{Caratheodory3}) is also of large measure near one.
Both results are proved by the standard argument in
Diophantine approximation theory by estimating the measure of the
excluded balls.

Later on, in Section~\ref{sec:improved}, we will show that the set $\G$ can be improved
to another set $\widetilde \G$, which is tangentially accessible in more points
(compare with Proposition~\ref{prop:extension}).

\begin{definition}\label{tangentiallyaccessible}
Let $\C$ be a complex domain. We say that a point $\lambda_0 \in  \C$ is
tangentially accessible in $\C$, when there exists a unit complex
number $u$ (denote by $\bar u$ its complex conjugate)
and there exist $\delta > 0$, $\gamma > 0$,  $m \ge 2$, such that
\[
\Gamma(\lambda_0, m;\gamma, \delta)\equiv \{ \lambda_0 + t u + s \bar u \,
:\ |t|, |s| <  \delta,\quad s \ge \gamma |t|^m \}
\subset \C\ .
\]
We say that a point is tangentially accessible from both sides  in $\C$,
when
\[
\Gamma^+(\lambda_0, m;\gamma, \delta) \equiv \{ \lambda_0 + t u + s \bar u \,
:\ |t|, |s| <  \delta,\quad |s| \ge \gamma |t|^m \}
\subset \C\ .
\]
When we need to me more precise, we can talk about an $m$-tangentially
accessible point.
\end{definition}

Note that, clearly if a point $\lambda \in \complex$
 is tangentially accessible for
a set $\Upsilon$ and $\Upsilon \subset \Lambda$, then
$\lambda$ is tangentially accessible for $\Lambda$.

The property of being tangentially accessible is, of course, only relevant for the
points in the boundary. It means that we can get regions  bounded
by parabolas tangent to the point inside the domain.
If the boundary is
given by a differentiable curve, all the points are tangentially accessible.
Also, when we transform a domain by a differentiable mapping, all the
tangentially accessible points for the original domain get
mapped into tangentially accessible points for the image.

The fact that a point is tangentially accessible has important
consequences. For example, in
\cite{Sokal,Hardy} it is shown that the asymptotic expansions based at that point can
be Borel summed and determine uniquely the function in the sector;
the book \cite{Caratheodory2} contains several other properties, which are a consequence
of accessibility.

In our case, the points in the boundary of $\Lambda$ are tangentially accessible from
both sides. We also note that the asymptotic expansions constructed
in Section~\ref{sec:Lindstedt} are defined on both sides of the domain. It then
follows that we can continue the functions in a unique way across
these points.

\begin{proposition}\label{pro:tang}
Consider  $\omega  \in \D_d(\nu, \sigma )$ and assume that for
some $m\geq 2$ the following inequality holds:
\begin{equation}\label{sigmacondition}
\sigma > m d\ .
\end{equation}

Then, almost all points in the unit circle are $m$-tangentially accessible for $ \cup_A \Upsilon(A; \omega,  \sigma)$.

In particular, by taking $A$ sufficiently small, we can get a set of $m$-tangentially
accessible points, whose complement has measure as small as desired.
\end{proposition}


\begin{proof}
Remember that an upper bound for the  set $\Upsilon$ is obtained by removing balls centered at
$e^{2 \pi i \omega \cdot k}$ of radius approximately equal to $A^{-1} |k|^{-\sigma}$.

We observe that for all the points in the unit
circle that are at a distance
bigger than
$C(\gamma^{-1}A^{-1} |k|^{-\sigma})^{1/m}$ from $e^{2 \pi i \omega \cdot k}$
for some positive constant $C$, we can find a domain of
the form $\Gamma(\alpha, m)$, which does not touch the excluded
ball centered at $e^{2 \pi i \omega\cdot k}$.

Therefore, for all points in a subset of the unit circle whose complement has measure less
than $C(\gamma^{-1} A^{-1} |k|^{-\sigma})^{1/m}$, the condition
imposed by the ball corresponding to $k$ is not an impediment for
being $m$-tangentially accessible from both sides.

It then follows that the set of points that satisfy the conditions
imposed for all $k$ to be $m$-tangentially accessible from both sides
has a complement whose measure is less than
\[
\sum_{k \in \integer^d \setminus \{0\}}
2 C(\gamma^{-1} A^{-1} |k|^{-\sigma})^{1/m} \le C' \gamma^{-1/m} A^{-1/m}  \sum_{\ell \in \nat}
\ell^{ d -1 - \sigma/m}
\]
for some constant $C'$; the factor 2 at the l.h.s. takes into account
that we consider both sides of tangential accessibility.  We see that
under the condition \eqref{sigmacondition}, the sum above is finite
and, by choosing the constant $\gamma$ large enough, we can obtain
that the complement of the $m$-tangentially accessible points in
$\Lambda(A;\omega)$ has measure  as small as desired.
\end{proof}

\begin{proposition} \label{density}
Assume that $\omega \in \D_d(\nu, \tau)$ with $2\tau > d$ and that $N
\ge 1$.  The point $\lambda = 1$ is a point of density for the set
$\Lambda \subset \complex$ (with the two-dimensional Lebesgue
measure).  If $\tau > d $, then $1$ is a point of density for
$\Lambda \cap \mathbb{S}^1$ (with the one-dimensional Lebesgue
measure).
\end{proposition}

\begin{proof}
This is a standard excluded measure argument. Fix $\rho>0$ sufficiently small.
In the following, we will not specify the constants that we generically denote as $C$.
Let $\C_\rho=\{\lambda\in\complex\,:\ \rho < | \lambda -1 | < 2 \rho\}$.

The set $\Lambda^c \cup \C_\rho$   is the union of the sets
\[
\R_{k,\rho} = \{\lambda \in \complex\,:\  |e^{2 \pi i k\cdot\omega} - \lambda|  <
A^{-1} |k|^{-\tau}|\lambda-1|^{N+1},\  \rho < | \lambda -1 | < 2 \rho \}\ .
\]
Note that, in particular, if there is a point $\lambda$ satisfying the
two conditions defining $\R_{k,\rho}$, we have
\[
|e^{2 \pi i k \cdot \omega} -1  | \le A^{-1} |k|^{-\tau}\ 2^{N+1}\ \rho^{N+1} +  2 \rho \le 3 \rho\ ,
\]
where the last inequality holds for $\rho$ small enough. Since $\omega$ is Diophantine, this implies that
$|k| \ge C  \rho^{-{1\over \tau}}$.

Now, for a fixed $k$ we see that the set of points $\lambda$ for
which  $|e^{2 \pi i k \cdot \omega} - \lambda| \le 2^{N+1}\ A^{-1} \rho^{N+1} |k|^{-\tau} $
is a circle with area smaller than $C A^{-2} \rho^{2 (N+1)} |k|^{-2\tau}$.

Hence the total area excluded is less than:
\[
\begin{split}
\sum_{k\in\integer^d\backslash\{0\},|k| \ge C \rho^{-{1\over\tau}}}
C  A^{-2}  \rho^{2(N+1)} |k|^{-2 \tau}&\le
\sum_{j\in\integer,|j|  \ge C \rho^{-{1\over\tau}}}
C A^{-2}  \rho^{2(N+1)}  |j|^{-2\tau + d -1} \\
&\le C A^{-2} \rho^{2(N+1)}  \rho^{\frac{1}{\tau}( 2 \tau - d)}\ .
\end{split}
\]

We observe that, under the hypothesis $2\tau>d$
the measure of the regions excluded in an annulus of
inner radius $\rho$ and outer radius $2\rho$ is
less than
\[
C A^{-2} \rho^{2(N+1) + {1\over \tau} (2\tau -d)}\ .
\]
Hence, the excluded measure in the ball of radius $\rho$
around the origin is bounded by a power greater than $2$. Indeed
the power is arbitrarily large if $N$ is large enough.

The argument for the measure excluded in the unit circle is
extremely similar.  The only thing we have to change is to
estimate the length of the  excluded intervals rather than the
area of the balls.  We obtain
an estimate for the excluded length less than
\[
C A^{-2} \rho^{(N+1) + {1\over \tau} (\tau -d)}\ .
\]

\end{proof}

\section{A conjecture on the optimality of the results}\label{sec:optimal}

The domain of analyticity of the embedding $K_\eps$ and the parameter
$\mu_\eps$ can be optimized in a variety of different ways; for
example, the constants in the cohomology equations can be sharpened,
the order of the expansion can be taken to optimize the results,
by exercising more care in the method presented here or, perhaps,
developing a new method, etc.
Indeed in Section~\ref{sec:improved} we will present an improved domain
of analyticity.
Nevertheless, in this Section we want to
argue that the shape of the domains established in
Section~\ref{geometry} is essentially optimal in the following sense.

\begin{conjecture}\label{isoptimal}

Consider a generic family of mappings $f_{\mu,\eps}$, satisfying the
hypotheses of Theorem~\ref{main}.  Consider the solution
$(K_\eps,\mu_\eps)$ produced in Theorem~\ref{main} and the maximal
domain where such solution is defined. Then, there exists a sequence
of balls $\B_k$ in the complex plane, which are not included in the
maximal domain of analyticity.

The balls $\B_k$ with centers in $\eps_k$, such that $\lambda(\eps_k)=e^{2\pi ik\cdot \omega}$,
correspond to the balls excluded in the definition of $\G$ (see \equ{goodset}).
\end{conjecture}

We note that if Conjecture~\ref{isoptimal} were
true, it would have important consequences for the analytic
properties of the asymptotic expansions.

The conjectured analyticity domains do not contain sectors centered at
the origin with aperture bigger than $\pi/a$. When $a> 1$, we cannot
use the Phragm\'en-Lindel\"of principle (\cite{PhragmenL08},
\cite{SaksZ65}, \cite{Hardy}) to obtain that the Taylor asymptotic expansion
determines the function. Indeed, one could get non-trivial analytic
functions with zero asymptotic expansion. It is not clear whether any
method of summation based only on the expansion produces the correct
result solving the functional equation.

The main reason for stating Conjecture~\ref{isoptimal} is an
argument which goes by contradiction, already used in \cite{CCCL2015}
to which we refer for more details.  We recall that the
drift $\mu$, together with the embedding, is an unknown of the
problem. Then, we extend the one-parameter family $f_{\mu,\eps}$, depending
on the parameter $\eps$, to a family
$f_{\mu,\eps,\gamma}$ depending also on a parameter $\gamma$, such
that $f_{\mu,\eps,0}=f_{\mu,\eps}$.

Then, we proceed as follows:

\begin{enumerate}
\item[$(i)$]
we assume that $f_{\mu,\eps,\gamma}$ is analytic in the parameters $\eps$, $\gamma$;
then, under suitable resonance conditions on the conformally symplectic factor
(see Proposition~\ref{destruction} below), that ensure that $\lambda(\eps_k)=e^{2\pi ik\cdot\omega}$ for
an infinite sequence $\eps_k$, $k\in\integer^d\backslash\{0\}$, we claim that there is no family of functions
analytic in $\gamma$ for $\gamma$ small near $\varepsilon_k$;
\item[$(ii)$]
by recalling a result stated in \cite{CCCL2015}, we show that if we have solutions $f_{\mu,\eps}$
in the space of analytic functions for all $f_{\mu,\eps}$ in a neighborhood,
there has to be a family $f_{\mu,\eps,\gamma}$ analytic in $\gamma$
for $\gamma$ small;
\item[$(iii)$]
the consequence of the above two statements is that
in every neighborhood, there has to be a family $f_{\mu,\eps}$,
which is analytic in the neighborhood of the resonant points where $\lambda(\eps_k)=e^{2\pi ik\cdot\omega}$.
\end{enumerate}

The result in $(i)$ is given by the following Proposition.


\begin{proposition}\label{destruction}
Let the family of mappings $f_{\mu,\eps}$ satisfy the hypotheses of Theorem~\ref{main}.
Assume that there exists an analytic solution $(K_\eps,\mu_\eps)$, solving the invariance
equation \equ{invariance}. Let $f_{\mu,\eps,\gamma}$ be a generic family, which is
analytic in $\eps$, $\gamma$, with conformally symplectic factor $\lambda=\lambda(\eps,\gamma)$ and
such that
$$
f_{\mu,\eps,0}=f_{\mu,\eps}\ .
$$
Then, we can find an infinite sequence $\eps_k$, $k\in\integer^d\backslash\{0\}$, with
$\lambda(\eps_k)=e^{2\pi ik\cdot \omega}$, such that there is no formal expansion in
$\gamma$ for the solution $(K_{\eps_k,\gamma},\mu_{\eps_k,\gamma})$ of \equ{invariance} associated to $f_{\mu,\eps,\gamma}$.
\end{proposition}

The proof of Proposition~\ref{destruction} is given in Section~\ref{sec:proofdes}.\\

Notice that the proof of Proposition~\ref{destruction} shows that near $\eps=\eps_k$,
the effects of a change on $(K_\eps,\mu_\eps)$ are much larger than the change on $f_{\mu_\eps,\eps}$.
Hence, it seems likely that one can destroy the solution $(K_\eps,\mu_\eps)$  without altering too much
the family $f_{\mu_\eps,\eps}$.

We interpret Proposition~\ref{destruction} as meaning that the solutions for
families which have large domains of analyticity are unstable and can be easily destroyed.
Of course, the sense in which we prove instability is somewhat weaker than what is needed to
reach the conclusions of Conjecture~\ref{isoptimal} rigorously, but it goes in the right direction and this motivates the
formulation of the results as a conjecture.

We note that the above argument suggests that the set of functions
$f_{\mu,\eps,\gamma}$, such that $f_{\mu,\eps,\gamma}$ has a torus
that is analytic in a neighborhood of $\eps_k$, is a set of infinite
co-dimension (in particular nowhere dense). Once this result were
established, the set on which infinitely many resonances are destroyed
is residual, in the sense of Baire category.

Note that our argument in Proposition~\ref{destruction} is somewhat
similar to the arguments in \cite{Poincarefrench} about the lack of
uniform integrability; indeed, what we conjectured is
an analogue of
the lack of integrability for generic systems, based on the lack of uniform
integrability.

\subsection{Proof of Proposition~\ref{destruction} }\label{sec:proofdes}

The proof of Proposition~\ref{destruction} is very similar to the results we obtained on
Lindstedt series in Section~\ref{sec:Lindstedt}. If there existed
an approximate solution $(K_\eps,\mu_\eps)$ satisfying
$$
f_{\mu_\eps,\eps,\gamma}\circ K_\eps-K_\eps\circ T_\omega=E
$$
with an error $E$, then we can find a matrix $M$ such that
$$
Df_{\mu_\eps,\eps,\gamma}\circ K (\theta) M(\theta)=M(\theta+\omega)\
\left(%
\begin{array}{cc}
  \Id & S(\theta) \\
  0 & \lambda\Id \\
 \end{array}%
\right)+R\ ,
$$
where $M$ and $S$ have been defined, respectively, in \equ{Mdef} and \equ{Sdef}
with $f_{\mu_a,\eps}$ replaced by $f_{\mu_\eps,\eps,\gamma}$, and $R$ is the error in
the automatic reducibility as in \equ{reduction}.
By the procedure of automatic reducibility, we see that after making expansions in $\gamma$,
we have to solve the following two equations:
\beqa{eqgamma}
W_1-W_1\circ T_\omega&=&-S\, W_2-\widetilde E_1-\widetilde A_1\sigma\nonumber\\
\lambda W_2-W_2\circ T_\omega&=&-\widetilde E_2-\widetilde A_2\sigma\ ,
\eeqa
where $W=(W_1,W_2)$ and $\sigma$ are the corrections to the step,
$\widetilde E\equiv (M\circ T_\omega)^{-1}E$, $\widetilde A=[\widetilde A_1|\widetilde A_2]\equiv (M\circ T_\omega)^{-1} D_\mu f_{\mu_\eps,\eps,\gamma}\circ K_\eps$.
We must require that the right hand sides of \equ{eqgamma} have zero average.
This can be obtained by properly choosing $\sigma$, provided that the non-degeneracy condition \equ{non-degeneracy} is satisfied.

We note that precisely at the points $\eps_k$ where $\lambda(\eps_k)=e^{2\pi ik\cdot\omega}$,
$k\in\integer^d\backslash\{0\}$, then the second equation in \equ{eqgamma}
has obstructions to solutions.
We also note that, because of the form pointed out in \equ{eqgamma}, the solution of the order $j$ equations
of a formal expansion in $\gamma$ exists
only if ${d\over {d\gamma}} f_{\mu_\eps,\eps,\gamma}|_{\gamma=0}$ satisfies a condition.
Of course the conditions for different $j$ are independent, since they affect different coefficients.
Therefore, the existence of expansions to order $j$ requires that the mapping belongs to a co-dimension $j$
submanifold of maps. \qed

\vskip.1in

As indicated in the preliminaries, we note that for the values
$\eps\simeq \eps_k$, we have that the changes on $K_\eps$,
$\mu_\eps$ induced by perturbations of $f_{\mu_\eps,\eps,\gamma}$
are much larger than the perturbations of
$f_{\mu_\eps,\eps,\gamma}$ itself, this makes it plausible that
one can introduce changes in $f_{\mu_\eps,\eps}$ which destroy the
analytic $K_\eps$ without destroying $f_{\mu_\eps,\eps}$ (see
\cite{Siegel54} for similar arguments).

\section{Automatic reducibility and Lindstedt series}\label{fastlindstedt}

In this Section we present several results about Lindstedt series. First, we show that we can obtain Lindstedt
series expansions around any point $\eps_0$ (see Section~\ref{sec:series}).
Some relations with the theory of monogenic functions are presented in Section~\ref{sec:monogenic}.
Then, by lifting the automatic reducibility
to families, we show how to get quadratically convergent algorithms for the Lindstedt series; as a byproduct of this result,
we will obtain an alternative proof of Part B of Theorem~\ref{main} (see Section~\ref{sec:quadratic}).
Finally, we will show how the results of Section~\ref{sec:series} lead to define an improved domain
of analyticity (see Section~\ref{sec:improved}) and we conclude by establishing
 the Whitney differentiability
of $K_\eps$, $\mu_\eps$ on $\G$ (see Section~\ref{sec:whitney}).

\subsection{Lindstedt series from any analytic torus}\label{sec:series}
In this Section we show that if for some $\eps_0$ there
are  $K_{\eps_0},\mu_{\eps_0}$ solving
the invariance equation at $\eps_0$, and which satisfy some
mild non-degeneracy assumptions (which are implied by
$|\eps_0|\ll 1$),  we
can find a formal  power series in $\eps - \eps_0$ that solves
the invariance equation in the sense of formal power series expansions.

Of course, for points in the interior of the analyticity domain,
the existence of expansions is obvious.  The interesting case is that the same result holds for
some points in the boundary of the analyticity domain. As we will see,
this matches very well with the theory of monogenic functions (see Section~\ref{sec:monogenic}).
It is also important to notice that even in the interior of the domain of analyticity
the present method gives very good estimates of the formal power series,
much better than what can be obtained from just Cauchy estimates,
see Remark~\ref{rem:deriv}.

\begin{proposition} \label{Lindstedt_anywhere}
Let $\omega$ be a Diophantine vector in the sense of
Definition~\ref{vectordiophantine}, let
$f_{\mu,\eps}$ be a family of
conformally symplectic systems as before.
Assume that for some $\eps_0$ we can find $K_{\eps_0} \in \A_\rho$,
$\mu_{\eps_0} \in \complex^d$ such that
$f_{\mu_{\eps_0},\eps_0}^*\Omega = \lambda(\eps_0) \Omega$ and
\beq{inveps0}
f_{\mu_{\eps_0},\eps_0}\circ K_{\eps_0} = K_{\eps_0}\circ T_\omega \ .
\eeq
Assume furthermore that $\lambda(\eps_0)$ is Diophantine with respect to
$\omega$ in the sense of Definition~\ref{lambdadiophantine},
that $K_{\eps_0}(\torus^d_\rho)$ is well inside the domain of
definition of $f_{\mu_{\eps_0}, \eps_0}$ and that $K_{\eps_0}$ satisfies
the non-degeneracy condition {\bf H3} of Theorem~\ref{mainKAM}.

Then, for any $0 < \rho' < \rho$, there is a formal Lindstedt power series solution
$K_\eps^{[\infty]}$, $\mu_\eps^{[\infty]}$:
\begin{equation} \label{Lindstedt_arbitrary}
\begin{split}
\mu_\eps^{[\infty]} &= \sum_{n=0}^\infty \mu_n (\eps- \eps_0)^n\ , \\
K_\eps^{[\infty]} &= \sum_{n=0}^\infty K_n (\eps- \eps_0)^n\\
\end{split}
\end{equation}
with coefficients $K_n \in \A_{\rho'}$, $\mu_n\in\complex^d$, that satisfy
the invariance equation in the sense of formal power series,
namely
\begin{equation} \label{reminder}
\left \| f_{\mu_{\eps}^{[\le N]},\eps_0} \circ K^{[\le N]}_\eps -  K^{[\le N]}_{\eps}\circ T_\omega \right \|_{\rho'} \le C_N |\eps - \eps_0|^{N+1}\ ,
\end{equation}
where $K_\eps^{[\leq N]}=\sum_{n=0}^N K_n(\eps-\eps_0)^n$, $\mu_\eps^{[\leq N]}=\sum_{n=0}^N \mu_n(\eps-\eps_0)^n$.

The coefficients may be chosen so that the normalization \eqref{normalized},
with $M$ the corresponding for $K_{\eps_0}$,
is satisfied in the sense of formal power series.
The series that satisfy
\eqref{inveps0} and \eqref{normalized} in the sense of power series are
unique.
\end{proposition}

\begin{proof}
Introduce
the matrix $M$ corresponding to $K_{\eps_0}$ as in \eqref{Mdef}, for which we have
\beq{ast}
D f_{\mu_{\eps_0}, \eps_0}\circ K_{\eps_0} M(\theta)  =  M(\theta + \omega)
\ \left({\begin{array}{cc}
  \Id & S(\theta) \\
  0 & \lambda(\eps_0) \Id \\
 \end{array}%
 }\right).
\eeq
Since the invariance equation  has zero error, so does the
reducibility equation \equ{ast}.

If we substitute \eqref{Lindstedt_arbitrary} in the invariance equation and
equate the coefficients of order
$(\eps -\eps_0)^n$ on both sides,
we obtain for  $n \ge 1$:
\begin{equation}\label{recursion}
D f_{\mu_{\eps_0}, \eps_0}\circ K_{\eps_0} K_n
+ (\partial_\mu f)_{\mu_{\eps_0}, \eps_0}\circ K_{\eps_0}  \mu_n   +
R_n = K_n\circ T_\omega
\ ,
\end{equation}
where $R_n$ is a polynomial expression in $K_1$, $\ldots$ $K_{n-1}$,
$\mu_1$, $\ldots$ $\mu_{n-1}$ with coefficients which are derivatives
of $f$ evaluated at $K_{\eps_0}, \mu_{\eps_0}$.

We think of \eqref{recursion} as a recursion that allows one to compute
$K_n$, $\mu_n$, once $K_1,\ldots K_{n-1}$, $\mu_1,\ldots \mu_{n-1}$ have been computed.

We also substitute \eqref{Lindstedt_arbitrary} in the normalization condition
\eqref{normalized} so that we can obtain expansions of functions that
satisfy the normalization condition with respect to $K_{\eps_0}$.  We obtain
that the coefficient of order $(\eps -\eps_0)^n$ of the normalization
is just
\begin{equation}\label{normalizationn}
\int_{\torus^d} \Big[M^{-1} (\theta) K_n(\theta) \Big]_1 \, d\theta = 0\ .
\end{equation}

Of course, the equation \eqref{recursion} is
a particular case of the
equations appearing when we studied the Newton's equation
for approximately invariant solutions in Section~\ref{sec:theoremmainKAMproof}
(it suffices to take $E = R_n$).
Many more
details can be read in \cite{CallejaCL11}.  We note that using the
change of variables $K_n = M W$ reduces the equation \equ{recursion}
to constant coefficients cohomology equations. We emphasize that the series whose coefficients
satisfy \equ{normalizationn} are unique.
\end{proof}

Two important corollaries of Proposition~\ref{Lindstedt_anywhere} are given below.
Corollary~\ref{cor:monogenic} shows that one can construct asymptotic series of the form
\eqref{Lindstedt_arbitrary}, which in principle do not converge, such that the truncated
series to order $N$ satisfies the invariance equation up to a small error. This result
allows to start an iterative process from the approximate solution,
leading to an exact solution which is locally unique.

Corollary~\ref{derivatives_bounded} shows that the derivatives of $K_\eps$, $\mu_\eps$ are
obtained by finding the expansions; their estimates involve only the loss of domain in
the coordinates, while they are uniform in $\eps$.

\begin{cor} \label{cor:monogenic}
Within the assumptions of Proposition~\ref{Lindstedt_anywhere}, let $\eps_0$ be such that we can
find $K_{\eps_0}\in\A_\rho$, $\mu_{\eps_0}\in\complex^d$ satisfying \equ{inveps0}. For any
$0<\rho'<\rho$,
let $K_{\eps}^{[\infty]}$, $\mu_\eps^{[\infty]}$be such that \equ{Lindstedt_arbitrary} holds
with coefficients $K_n \in \A_{\rho'}$, $\mu_n\in\complex^d$ and that the truncated series
$K_\eps^{[\leq N]}$, $\mu_\eps^{[\leq N]}$ satisfy \equ{reminder}
and the normalization condition.

Then, for all $\eps$ sufficiently close to $\eps_0$, one has
\beqa{equbounds}
\left\| \sum^N_{n= 0}   K_n(\eps  -\eps_0)^n  - K_\eps \right\|_{\rho'} &\le&
C_N |\eps - \eps_0|^{N+1}  \nonumber\\
\left| \sum^N_{n= 0}  \mu_n (\eps  -\eps_0)^n  - \mu_\eps \right| &\le&
C_N |\eps - \eps_0|^{N+1}
\eeqa
for some positive constant $C_N$ depending on $\rho'$.

\end{cor}

\begin{proof}
From the construction of the series we obtain that the polynomials
$K_\eps^{[\leq N]}$, $\mu_\eps^{[\leq N]}$ satisfy the invariance equation
up to a small error as in \equ{reminder}. Hence, all assumptions {\bf H1}-{\bf H5}
of Theorem~\ref{mainKAM} are satisfied. Applying Theorem~\ref{mainKAM} with
$K_a=K_\eps^{[\leq N]}$, $\mu_a=\mu_\eps^{[\leq N]}$, we obtain an exact
solution $K_\eps$, $\mu_\eps$, which satisfies the bounds \equ{equbounds}.
Using the local uniqueness of the normalized solutions in
Lemma~\ref{lem:unique}, we obtain that
the solution produced is $K_\eps$, $\mu_\eps$.
\end{proof}

\begin{cor} \label{derivatives_bounded}
Assume that for $\eps \in \G$ we have
\beqano
\|K_\eps\|_\rho &\le& B,\nonumber\\
 | \mu_\eps| &\le& B
\eeqano
for some $B > 0$. Then, we have for $ 0 < \rho' < \rho$:
\beqa{der}
\left\|\left( \frac{d}{d \eps}\right)^j K_\eps\right\|_{\rho'} &\le&   C_j ( \rho -\rho')^{-(2 \tau+2d) j} B\nonumber\\
 \left|\left( \frac{d}{d \eps}\right)^j\mu_\eps\right|&\le&  C_j ( \rho -\rho')^{-(2 \tau+2d) j} B\ ,
\eeqa
where the derivative is understood in the regular sense for
$\eps$ in the interior of $\G_{r_0}$ and as the (unique!) term determined from the Lindstedt
expansion; the constant $C_j$ depends also on the Diophantine constant.
\end{cor}

\begin{proof}
At each step, the derivatives in \equ{der} are obtained by solving
two cohomology equations: one of them is a regular cohomology equation not involving the conformal
factor, while the other depends on $\lambda$ (see Proposition~\ref{Lindstedt_anywhere} and compare, e.g.,
with \equ{eqgamma}). At each step one has a
constant loss of domain $(\rho-\rho')$.
The solution of each equation gives a factor which contains a power of the loss of domain $(\rho-\rho')^{-(\tau+d)}$ as stated in
Lemma~\ref{cohomology}.
The constant $C_j$ depends also on the quantities appearing in the definition of the set
$\G$ in \equ{goodset} and precisely on the product $\nu(\omega;\tau)\ \nu(\lambda;\omega,\tau)$.
\end{proof}

Notice that we have stated that the bounds of $|\left( \frac{d}{d \eps}\right)^j\mu_\eps|$
depend on $\rho -\rho'$, because this is what comes directly from
the proof. Of course, $\mu_\eps$ does not depend on $\rho$, so
that we could simply get
\[
\left|\left( \frac{d}{d \eps}\right)^j\mu_\eps\right|\le
 C_j \rho^{- (2 \tau+2d) j} B\ .
\]

\begin{remark}\label{rem:deriv}
It is interesting to compare Corollary~\ref{derivatives_bounded}
with the results obtained by applying the
Cauchy bounds by drawing just a small circle.
Of course, Cauchy bounds produce bounds on the derivatives
in the same space $\A_\rho$ and the bound involves
the distance to the boundary. The bounds in Corollary~\ref{derivatives_bounded}
are uniform in the distance to the boundary,  but are only true in
a slightly weaker space.

The main reason for the difference in the results is that we use not only function theoretic properties,
but we use that the functions we are studying
satisfy a functional equation.
The use of the functional equation to obtain the derivatives
provides much better estimates than
those available just using the function theoretic properties, albeit in a different space.

Of course, there are many functions that are defined in the same domains,
but which, not satisfying any functional equation, have much worse properties.
The systematic use of the functional equation and the fact that this functional equation  leads to
automatic reducibility is one of the big advantages of the present method
with respect to using only function theoretic methods.
\end{remark}

\subsection{A quadratically convergent algorithm for Lindstedt series and an alternative proof of Part B of Theorem~\ref{main}}\label{sec:quadratic}

An alternative way to compute the step by step series as in part A of
Theorem~\ref{main} is to lift Algorithm~\ref{alg:step} to the family
of maps depending also on the parameter $\eps$.  Furthermore, the lift
of the algorithm leads to a proof of Part B of Theorem~\ref{main}
based on an implementation of Newton's method after introducing a lift
to functions of $\eps$ and performing a Newton's step as a function of
$\theta$ and $\eps$.

In the following, we show that lifting the algorithm to compute
families of functions depending on the parameter $\eps$ is a very
natural procedure.
We consider the family $f_{\mu,\eps}$ of conformally symplectic maps. Assume that
for some $\eps$ such that $|\eps - \eps_0|$ is sufficiently small, we start with
an approximate solution
\begin{equation} \label{expansionOrdN}
K_\eps^{[\leq N]} = \sum_{n=0}^{N} K_n (\eps -\eps_0)^n\ , \qquad
\mu_\eps^{[\leq N]} = \sum_{n=0}^{N} \mu_n (\eps -\eps_0)^n\ ,
\end{equation}
which satisfies the invariance equation up to order $N+1$, i.e.
\beq{invfast}
f_{\mu^{[\leq N]}_{\eps},\eps}\circ K_{\eps}^{[\leq N]}-K_{\eps}^{[\leq N]}\circ T_\omega\,= \,E_\eps^{[\leq N]},
\eeq
with the following bound on the error for some $0<\rho'<\rho$:
\[\|E^{[\leq N]}_\eps\|_{\rho'} \leq C |\eps-\eps_0|^{N+1}\ . \]
Let us also define:
$$
\lambda_\eps^{[\leq N]} = \sum_{n=0}^{N} \lambda_n (\eps -\eps_0)^n\ .
$$
Algebraic manipulations similar to those performed in the proof of
Theorem~\ref{mainKAM} lead to the expression:
\beq{R0}
Df_{\mu^{[\leq N]}_{\eps},\eps}\circ K_{\eps}^{[\leq N]}\ M^{[\leq N]}_\eps(\theta)=M_\eps^{[\leq N]}(\theta+\omega)\
\left({\begin{array}{cc}
  \Id & S_\eps^{[\leq N]}(\theta) \\
  0 & \lambda_\eps^{[\leq N]}\Id \\
 \end{array}%
 }\right)+R(\theta)
 \eeq
for a suitable error function $R=R(\theta)$, where (again) we define the functions
$M_\eps^{[\leq N]}$,
$N_\eps^{[\leq N]}$ as
$$
M_\eps(\theta)\equiv[DK_{\eps}^{[\leq N]}(\theta)\,|\,J^{-1}\circ K_{\eps}^{[\leq N]}(\theta) DK_{\eps}^{[\leq N]}
(\theta)N_\eps^{[\leq N]}(\theta)]
$$
with
$$
N_\eps^{[\leq N]}(\theta)\equiv (DK_{\eps}^{[\leq N]\, \top}(\theta)\,\, DK_{\eps}^{[\leq N]}(\theta))^{-1}
$$
and $S_\eps^{[\leq N]}$ as in \equ{Sdef}.
We look for a new approximate solution, given by the updating of the approximate values
$K_{\eps}^{[\leq N]}$ and $\mu_{\eps}^{[\leq N]}$ with corrections
$\Delta_\eps$ (equivalently $W_\eps$) and $\sigma_\eps$, i.e.,
 $K'_\eps=K_{\eps}^{[\leq N]}+\Delta_\eps=K_{\eps}^{[\leq N]}+M_\eps^{[\leq N]}W_\varepsilon$, $\mu_\eps'=\mu_{\eps}^{[\leq N]}+\sigma_\eps$.
Proceeding in the same way as in Section~\ref{sec:KAM}, we look
for the corrections $(W_\eps,\sigma_\eps)$, that satisfy, \beqano
\left({\begin{array}{cc}
  \Id & S_\eps^{[\leq N]}(\theta) \\
  0 & \lambda_\eps^{[\leq N]}\Id \\
 \end{array}%
 }\right)W_\eps-W_\eps\circ T_\omega=&-&(M_\eps^{[\leq N]}\circ T_\omega)^{-1}E_\eps^{[\leq N]}\nonumber\\
 &-&(M_\eps^{[\leq N]}
\circ T_\omega)^{-1}(D_\mu f_{\mu_{\eps}^{[\leq N]},{\eps}}\circ K_{\eps}^{[\leq N]})\sigma_\eps\ ,
\eeqano
which in components provides the following cohomological equations similar to \equ{c12}:
\beqa{WW}
W_{1}-W_{1}\circ T_\omega&=&-S_\eps^{[\leq N]}\, W_{2}
-\widetilde E_1-\widetilde A_1 \sigma\nonumber\\
\lambda_\eps^{[\leq N]} W_{2}-W_{2}\circ T_\omega&=&-\widetilde E_2
-\widetilde A_2 \sigma\ ,
\eeqa
where $\widetilde E\equiv (\widetilde E_1,\widetilde E_2)
=(M_\eps^{[\leq N]}\circ T_\omega)^{-1}E_\eps^{[\leq N]}$,
$\widetilde A\equiv [\widetilde A_1|\widetilde A_2]=(M_\eps^{[\leq N]}\circ T_\omega)^{-1}
(D_\mu f_{\mu_{\eps}^{[\leq N]},{\eps}}\circ K_{\eps}^{[\leq N]})$.
The functions $\widetilde E$, $\widetilde A$, $W_1$, and $W_2$ depend on the parameter
$\eps$, but we omit it in the notation to avoid cluttering.

We notice that the cohomological equations in \equ{WW} can be
solved under the analogous of the
non--degeneracy condition \equ{non-degeneracy}. Then,
the solution of \equ{WW} provides the desired corrections $(W_\eps,\sigma_\eps)$.
We emphasize that the objects involved in the solution of the cohomological equations
in \equ{WW}
are functions of $\eps$ and that the estimates used
for the iterative step in Algorithm~\ref{alg:step} are uniform in $\eps$.
Furthermore, if we use the supremum norm the same estimates for the iterative step
hold for the convergence in the families (see \cite{LlaveO00}).
Of course, with this sup--norm the
convergence is uniform for the whole $\eps$ dependent family.
Since the uniform limit of analytic functions is analytic, the solution must
depend analytically on $\eps$.

The iterative method also suggests a practical algorithm for the computation
of the perturbative expansion in $\eps$. If we start the iterative procedure
from an approximate solution having terms up to $(\eps - \eps_0)$ to the
$N$-th power,
the quadratic convergence of the Newton's method implies that
after every step the number of terms in the expansion doubles.
One can use methods in automatic differentiation
(see \cite{Haro, Brent, Biscani}) to implement the operators involved in the iterative step,
so at every iteration the new step produces $N$ new terms and
it also has the advantage that it corrects the previously
computed lower order terms. Furthermore, on top of being
quadratically convergent the Newton's method is numerically stable.

Moreover, if $\lambda$ satisfies the Diophantine assumptions,
within the regions in $\eps$ where we have good estimates
one can find a sequence of analytic functions converging uniformly in the domain $\G_{r_0}$.
In the set $\G_{r_0}$ we obtain the
convergence of the sequence of approximate solutions,
obtained applying the iterative steps, to the exact
solution with bounds analogous to those in \equ{Kmu}.
Hence, the fast algorithm described before provides also a different
proof of Theorem~\ref{main}.

\subsection{Relations of the previous results  with the theory of monogenic functions}
\label{sec:monogenic}

We recall that a  function $k$ defined on a set $\H \subset
\complex$ is monogenic, when for all $\eps_0$ in $\H$, the limit
\begin{equation} \label{quotientlimit}
\lim_{\eps \to \eps_0} \frac{ k(\eps) - k(\eps_0)}{\eps - \eps_0}
\end{equation}
exists.

When $\H$ is an open set, this is the nowadays current definition of differentiable
functions, but for more complicated sets, since the limit in
\eqref{quotientlimit} is taken only when $\eps \in \H$, the notion may be
very different. The book \cite{Borel17} contains an account of the history
of monogenic functions and their relation with the nowadays
standard notion of differentiability.
A problem of great interest is to study the relation
between monogenic functions in some sets and quasi-analytic classes
(that is, classes of functions such that the function is
determined by its values in a set). The modern standard result
establishing that analytic functions in an open set are determined by
their values on a set with an accumulation point is a predecessor
of those results in which the structure of the set
determines the function. It is also known that the monogenic
properties in domains with enough aperture lead to consequences
for the properties of asymptotic expansions (\cite{Winkler93}).

The result stated in Proposition~\ref{Lindstedt_anywhere}
implies in the classical language that
the functions $K_\eps, \mu_\eps$
are monogenic at many points
(namely, the points for which
$\lambda(\eps)$ is Diophantine with respect to $\omega$)
in the set $\G$.

The relations of KAM theory with the theory of monogenic functions has
been studied in \cite{Kolmogorov57, Arnold65, Herman85b, MarmiS03, MarmiS11, CarminatiMS14}. These
papers investigate the monogenic properties of the solutions, when the frequency
$\omega$ goes into the complex. When $\omega$ has complex coefficients, then
$e^{2 \pi i k \cdot \omega}$ may be bounded away from $1$. Hence, the cohomology
equation  does not present any small divisors and, indeed,
it may be a compact operator.

The approach of the above papers to study the monogenic properties is
very different from ours. They study in great detail the monogenic
properties of the solutions of linearized equation and use this as
the basis of an iterative method. In our case, we just use the
formal power series and the a-posteriori format of Theorem~\ref{mainKAM}.
We take advantage of the non-perturbative properties of Theorem~\ref{mainKAM} and the
fact that the solutions are characterized by solving a functional equation.

\subsection{An improved domain of analyticity}\label{sec:improved}
In this Section we will show that the domain $\G$ can be slightly improved. This does
not contradict Conjecture~\ref{isoptimal}, since the extensions constructed
here are precisely in the places away from the resonances which Conjecture~\ref{isoptimal}
claims are essential.

The following Proposition~\ref{prop:extension} is the main result
of this Section. We show that one can find a larger analyticity
domain for $K$, $\mu$ in such a way that the new domain is
$m$-tangentially accessible at any Diophantine point (the
analyticity of $K$ should be understood in a slightly smaller
domain).

The proof of Proposition~\ref{prop:extension}
is based
on Proposition~\ref{Lindstedt_anywhere} that
established  that we can make a Lindstedt expansion
around any $\eps_0$ sufficiently small provided
that $\lambda(\eps_0)$ is $\omega$-Diophantine.

\begin{proposition} \label{prop:extension}
Under the assumptions of Theorem~\ref{main} and Proposition~\ref{Lindstedt_anywhere}, let the functions
$K:\G_{r_0}\rightarrow \A_\rho$, $\mu:\G_{r_0}\rightarrow \complex^d$ solve \eqref{invariance}.

For any $0 < \rho' < \rho$, we can find an analyticity domain
$\widetilde\G$, such that any point $\eps_0$ sufficiently small is
$m$-tangentially accessible, $m\in \nat$, from both sides in $\widetilde \G$, provided $\eps_0$ is such that
$|\lambda(\eps_0)|=1$ and $\lambda(\eps_0)$ is Diophantine with respect to $\omega$.

The functions $K, \mu$ are analytic from $\G$ to
$\A_{\rho'}, \complex^d$, respectively.
\end{proposition}

\begin{proof}
To prove Proposition~\ref{prop:extension}, we observe that,
given any point $\eps_0$ as in the assumptions,
using  Proposition~\ref{Lindstedt_anywhere} on the existence of Lindstedt power series
in $\eps-\eps_0$, for any
$N$ we obtain a polynomial that solves the invariance
equation up to an error measured in the $\rho$-norm
which is smaller than $C_N | \eps - \eps_0|^{N+1}$.
Precisely, we construct the polynomials $K_\eps^{[\leq N]} = \sum_{n=0}^{N} K_n (\eps -\eps_0)^n$,
$\mu_\eps^{[\leq N]} = \sum_{n=0}^{N} \mu_n (\eps -\eps_0)^n$, such that
$$
\|f_{\mu_\eps^{[\leq N]},\eps}\circ K_{\eps}^{[\leq
N]}-K_{\eps}^{[\leq N]}\circ T_\omega\|_{\rho'} \leq
C_N|\eps-\eps_0|^{N+1}
$$
for some positive constant $C_N$.

We also observe that, according to Remark~\ref{rem:lambda},
when we move out of the unit circle, the number $\lambda$ becomes
$\omega$-Diophantine with a Diophantine constant which is
bounded from below by the distance to the circle. Precisely, if $|\lambda(\eps)|\not=1$, then we obtain
$$
\nu(\lambda(\eps);\omega,\tau)\leq \big|1-|\lambda(\eps)|\big|^{-1}\ .
$$
If we fix $0<\rho' < \rho$ and we denote $\delta = \rho - \rho'$,
we can apply Theorem~\ref{mainKAM}
provided that the error in {\bf H5} is sufficiently small, namely
\begin{equation}\label{mconditions}
C_N | \eps - \eps_0|^{N+1}\,  \big|1-|\lambda(\eps)|\big|^2 \leq C \nu(\omega;\tau)^2\ \delta^{4(\tau+d)}\ .
\end{equation}
We now observe that, because $\lambda(\eps)$ is a conformal
mapping which is at most tangent to $1$ of order $a$ due to {\bf H$\lambda$},
and $N$ is arbitrarily large, we can find a domain of $\eps$ $m$-tangentially
accessible ($m$ arbitrarily large) where \eqref{mconditions} is satisfied;
then, we can find an exact solution of the invariance equation \equ{invariance} for any $\eps\in\widetilde\G$.
Again we note that this solution can be obtained to be normalized with respect to
\equ{normalized} corresponding to $K_\eps$. Then, the local uniqueness of normalized
solutions tells that this should agree with the original solution.
\end{proof}

It is interesting to compare Proposition~\ref{prop:extension}
to Proposition~\ref{pro:tang}. Of course, they refer to (possibly) different sets:
Proposition~\ref{pro:tang} is based just on a measure theoretic argument
and only asserts the existence of points occupying a set of large measure,
whereas Proposition~\ref{prop:extension} gives a concrete criterion for
the points for which it applies. Of course, given the easy measure estimates of
Diophantine numbers, there must be points for which both propositions apply.

We note that the assumption that $\eps_0$ is sufficiently small can be
substituted with the requirement that there is a solution satisfying
the invariance equation.  We also note that the assumption
$|\lambda(\eps_0)| = 1$ is, of course, not needed.  When
$|\lambda(\eps_0)|\ne 1$, it is easy to show\footnote{Just note that
for all $\eps$ close to $\eps_0$, we have that
$|\lambda(\eps)|\ne  1$ and, in particular,
$\lambda(\eps)$ is Diophantine. For any $\eps$
close enough to $\eps_0$ we can take as initial point of the
iterative step $K_0, \mu_0$. } that we can get a solution defined in
a ball around $\eps_0$.

\begin{remark}
Even if Proposition~\ref{prop:extension} involves a loss of domain in
the $\theta$ variable, this is not a problem, because the regularity in
the parameterization variable can be bootstrapped. Indeed, in
\cite{CallejaCL11}, there is a result establishing that all
finitely differentiable solutions of the invariance equations are
analytic.
\end{remark}

\subsection{Whitney differentiability of the parameterization $K_\eps$
and the drift $\mu_\eps$ on $\G_{r_0}$}
\label{sec:whitney}

In this Section, we establish that the functions $K_\eps, \mu_\eps$
are Whitney differentiable in $\G_{r_0}$.
The argument we present here takes advantage of the fact that
$K_\eps, \mu_\eps$ solve a functional equation.
This allows us
to obtain expressions of the formal derivatives with respect to parameters,
since these formal derivatives satisfy functional equations which
we can solve using the automatic reducibility.
Using the a-posteriori  format of Theorem~\ref{mainKAM}, we obtain that the
formal expansions lead to a small remainder. The rest of the argument
uses that the domain $\G_{r_0}$ is a compensated domain, that means that
the points can be joined by a path in the interior whose length is comparable
to the distance between the points. There are many other similar
problems in KAM theory and it is possible that similar strategies may
work in them.

 Let us start by recalling
the classical  definition of Whitney differentiability
(\cite{Whitney36,Stein70,Fefferman09,Grafakos14}). We note that
this is a very real variable definition. In our case, when we are
considering functions of a complex variable,  we have, of course
that $\complex = \real^2$.

\begin{definition}\label{def:whitney}
Let $\D \subset \real^n$, $n\geq 1$, be a closed set. For $\ell\in\integer_+$, $\ell<p\leq \ell+1$, let $H$ be a Banach space.
We say that the function $h: \D \rightarrow H$
is $C^p$-Whitney differentiable (which we denote by
 $h \in C_{Wh}^p(\D)$),
when for every $x \in \D$ we can find
a collection of continuous functions indexed
by $k \in \nat^n$, say  $\{h^{(k)}_x \}_{0 \le |k| \leq \ell}$,
taking values in $X$, such that
\begin{equation}\label{whitneyderivatives1}
\|h^{(k)}_x \|\leq M\ ,\qquad   0 \le k \leq \ell,\ \ x\in \D
\end{equation}
and for all $x, x+\Delta\in \D$
\begin{equation}\label{Wcond}
\left \|h^{(k)}_{x +\Delta}-\sum_{|j+k|\leq\ell} {1\over {j!}} h^{(j+k)}_x\
\Delta^j \right\|\leq M\ \|\Delta\|^{p-|k|}\ .
\end{equation}
Then, the function $h$ is called $C^p$-Whitney differentiable
in $\D$ with Whitney derivatives $\{h_x^{(k)}\}_{|k|\leq \ell}$.
\end{definition}

If $\D = \real^n$, one can see that the case
$k = 0 $ of \eqref{Wcond} is just the Taylor expansion of
$h$, if $h$ was differentiable.
The cases $k \geq 1$ of \eqref{Wcond} are the Taylor expansions for
the derivatives.  It is known that when $\D = \real^n$, one needs to verify
only
the $k=0$ case of \eqref{Wcond}.
Then, the celebrated \emph{converse Taylor theorem}
(\cite{AbrahamR67, Nelson69}) shows that the function $h$ is indeed $p$-differentiable
 and that therefore, the $j$-th derivative is $(p-j)$-differentiable.

When $\D$ is a small subset of $\real^n$, the Whitney derivatives do not
need to be unique (think for example of the case when $\D$ is
a hyperplane; since \eqref{Wcond} only involves what happens in
the plane, we see that we can put anything we want in
the transverse directions).  Hence, in general, the Whitney derivatives are
not unique and one can see that  the case $k= 0$ in \eqref{Wcond}
does not imply the cases for higher $k$.

The remarkable result of \cite{Whitney36} is
that given a function $h$ which is $C^p$-Whitney differentiable in a closed set
$\D$, there exists a $C^p$-function (in the standard sense) in $\real^n$
which agrees with $h$ on $\D$.  The original result is presented only
for real valued functions, but it is easy to check the proof when the
functions take values in a
Banach space, since the proof follows from the proof in $\real^n$ line by line
(in contrast, the case when $\D$ is contained in a Banach space seems to be an open problem).

\begin{remark}
There are some well known extensions, which  we will not discuss here, but
which we mention. In \cite{Stein70}, one can find
results showing that the extension can be obtained to be
linear in $h$.  Actually, one can construct a linear  extension operator
$E_l$ which works for all $ p \in (l, l+1]$. It turns out that each different
$l$ requires a different construction. We note that there are
more sophisticated constructions, when the bounds  \eqref{Wcond} are
uniform on compact sets, but may be not uniform in $\D$.
The construction of \cite{Whitney36}
allows the case of infinite derivatives. It also shows that the extended
function is real analytic in the complement of $\D$. A survey of recent
developments in the problem appears in \cite{Fefferman09}.
\end{remark}

\begin{proposition} \label{whitney}
For given $\eps$ and under the hypotheses of Theorem~\ref{main}, we have
that the functions that produce the expansions of the normalized solution
of \eqref{invariance0}
given by Proposition~\ref{Lindstedt_anywhere}, say $K_n^{(\eps)}$ and $\mu_n^{(\eps)}$,
are continuous on $\G_{r_0}$ and
analytic in the  interior of $\G_{r_0}$, when we give the topology of
$\A_\rho  \times \complex$ to $K_n^{(\eps)}$ and $\mu_n^{(\eps)}$.

We also have that $K_n^{(\eps)}$ and $\mu_n^{(\eps)}$ are
$C^\infty$ in the sense of Whitney in $\G_{r_0}$, when we give to
the ranges the topology of $\A_{\rho'}\times \complex^d$ for any
$0< \rho' < \rho$.
\end{proposition}

\subsubsection{Proof of Proposition~\ref{whitney}}

Given one normalized solution, we have proved in
Proposition~\ref{Lindstedt_anywhere} that we can find
series expansions of the normalized
solutions  around any point in $\G_{r_0}$. These will be
the Whitney derivatives.

We note that we have produced the derivatives as polynomials in
the complex variable $\eps$. Of course, $\eps^j$ is a multilinear function
of degree $j$ of the real variables which are the components of
the complex number $\eps$.  We will use the notation of complex powers
to keep the notation we have used in the previous sections.

We will now verify that the components of
the  expansions $K_n^{(\eps)}, \mu_n^{(\eps)}$
 verify all the properties
in the Definition~\ref{def:whitney}.

To verify that  the functions $K_n^{(\eps)}$ , $\mu_n^{(\eps)}$ are continuous,
we just observe that the  $K_n^{(\eps)}$ are bounded in $\A_{\rho'}$.
Hence, by Montel's theorem \cite{SaksZ65} they lie in a precompact set
in $A_{\rho''}$ for any $0 < \rho'' < \rho'$. When we have a function
from a compact set to a compact set, continuity is equivalent to having
a closed graph. We can check that if we have a sequence $\eps_j \to \eps$
and $K_n^{(\eps_j)} \to K_n^{(\eps)}$, $\mu_n^{(\eps_j)}\to \mu_n^{(\eps)}$,
 we can see that the limiting functions $K_n^{(\eps)}$ , $\mu_n^{(\eps)}$ satisfy
\eqref{Lindstedt_arbitrary} and it is also normalized. Using
that the solution of \eqref{invariance0} is unique whenever it satisfies
the normalization
condition, we obtain that the solution is indeed the $K_n^{(\eps)}$.

An alternative proof of the continuity of
the expansions is obtained by induction. We have shown that
$K$ is continuous in $\eps$ and that the reducibility matrix is
continuous in $\eps$ in a slightly smaller domain in $\theta$. The expansion is
computed by recursion. It just requires to evaluate a polynomial
in the previous terms, multiplying by the automatic reducibility matrix,
solving the cohomology equation,
and applying other operations. Hence, by recursion, it is easy to
see that we obtain that the terms of the expansions are continuous,
provided that at every step we decrease (by an arbitrarily small
amount) the domain of analyticity in $\theta$.

We also note that  Corollary~\ref{cor:monogenic} shows
that the derivatives  satisfy the Taylor estimates. This is
the condition \eqref{Wcond} for $k = 0$. Of course, in
the interior points, we obtain that this is indeed the derivative
and, since it is a complex derivative, the function is analytic
in a neighborhood (this gives an independent proof of the analyticity
of the functions in the interior of $\G_{r_0}$).

The only thing that remains to verify in the
Definition~\ref{def:whitney} of
Whitney derivatives is the condition \eqref{Wcond} for $k> 0$. By the previous
remark, we could restrict ourselves to the case that one of the points
is close to the boundary; nevertheless, we will not need that.

As we mentioned before, the
condition for $k > 1$ in Definition~\ref{def:whitney},
 unfortunately, is not automatic. In \cite{LlaveO99, LiL10} one
can find examples where this fails.

The set we are considering, however, is much simpler than many other
closed sets. It is a connected set and indeed it has a big interior.
The key property is that the sets
we consider have the \emph{compensation} property, namely for any pair
of points in the domain, one can find a path joining them, whose
length is less or equal than a constant times the distance between the
points.
The set $\G_{r_0}$ satisfies a very strong form of the
compensation property as shown in Proposition~\ref{prop:compensation}
below.
\begin{proposition}\label{prop:compensation}
Given two points $\eps_1, \eps_2 \in \G_{r_0}$, there is a differentiable
path $\gamma$ starting in $\eps_1$, finishing in $\eps_2$ and contained in
$\G_{r_0}$, indeed, contained in
the interior of $\G_{r_0}$ except for the endpoints.  Moreover
$|\gamma|$, the length of $\gamma$, satisfies the inequality
\begin{equation} \label{eq:compensation}
|\gamma| \le \pi |\eps_1 - \eps_2|\ .
\end{equation}
\end{proposition}

\begin{proof}
Of course, if the domain $\G_{r_0}$ was a disk -- or any convex set --
the result would be trivial, since we could join two points
$\eps_1, \eps_2$ by a straight line which stays in the disk.

The only thing that we have to consider is the case that some of
the excluded circles from the boundary interrupts the straight line.

We consider the ball $\B$  centered at the boundary of $\{ \eps \in \complex
:|\lambda(\eps)|  \le 1 \}$ that passes
through $\eps_1, \eps_2$. This ball can be easily constructed, since the
center is the mediatrix between the points $\eps_1, \eps_2$ and the
boundary of the disk. This point is unique if we are considering
$A$ sufficiently small, because of the implicit function theorem.

Now we observe that any ball centered in the boundary which intersects
the segment $\eps_1 -\eps_2$ and which does not include any of the points
$\eps_1$, $\eps_2$ has to be contained in $\B$.

Then, to connect the two points we can move along the boundary of
the disk $\B$. Of course, being an arc of disk the path is bounded
by $\pi$ times the distance of the points (the bound is saturated when
the points are in opposite sides of a diameter).
\end{proof}

Using Proposition~\ref{prop:compensation}, the result of Proposition~\ref{whitney}
follows by a standard argument (see \cite{LiL10}).
Given $\eps_1,\eps_2\in\G_{r_0}$,
let $\gamma$ be a differentiable path contained in $\G_{r_0}$,
starting in $\eps_1$ and ending in $\eps_2$.
Using the fundamental theorem of calculus, we
write, as in the Lagrange version of Taylor's theorem:
\[
\begin{split}
K^{(\eps_2)}_n - K^{(\eps_1)}_n
&= \sum_{j = 1}^N \frac{1}{j!} \frac{d^j}{d\eps^j}
K_n^{(\eps_1)} (\eps_2 - \eps_1)^j\\
&+  \int_0^1 ds_1 \int_0^{s_1} ds_2 \, \cdots
\int_0^{s_{N-1}} ds_N\  \frac{d^{N+1}}{d\eps^{N+1}}
 K_n^{(\gamma(s_1 s_2 \cdots s_N))}
s_1 (s_1 s_2) \cdots (s_1 s_2\cdots s_{N-1})\\
&\qquad \gamma'(s_1) \gamma'(s_1 s_2)  \cdots \gamma'(s_1 s_2 \cdots s_N)\ .
\end{split}
\]

Since the path $\gamma$ is contained in the interior of the domain, we
have already shown that the terms of the Lindstedt expansions are the
complex derivatives of the function $K_\eps$ and therefore
that $K_\eps$ is analytic in the interior.  Hence, it follows that the
higher order derivatives are the derivatives of the lower
order derivatives and the functions
$K_n^{(\eps)}$ are  analytic in $\eps$.

Hence, we obtain
\[
\left\| K_n^{(\eps_2)} - K_n^{(\eps_1)} -
\sum_{j = 1}^N \frac{1}{j!} \frac{d^j}{d\eps^j}
K_n^{(\eps_1)} (\eps_2 - \eps_1)^j \right\|_{\rho'}
\le C |\eps_2 -\eps_1|^N\ ,
\]
which is the condition \eqref{Wcond} for Whitney differentiability.

The proof for $\mu$ goes along the same lines and is even more elementary.
\qed

\appendix

\section{The case of flows}\label{sec:flows}

In this Section, we extend the results of the paper to the case of flows. We recall that a family of vector fields
$\F_{\mu,\eps}$ is conformally symplectic, if there exists a function $\chi=\chi(\eps)$ with $\chi(0)=0$,
such that
$$
L_{\F_{\mu, \eps}} \Omega = \chi(\eps) \Omega\ ,
$$
(see \equ{fl}). The invariance equation for $(K_\eps, \mu_\eps)$ is the differential equation:
\beq{invf}
\partial_\omega K_\eps(\theta)=\F_{\mu_\eps,\eps}\circ K_\eps(\theta)\ ,
\eeq
where, for short, we denote the partial derivative as $\partial_\omega=\omega\cdot\partial_\theta$,
the unknowns are $K_\eps$, $\mu_{\eps}$.

We look for Lagrangian tori satisfying the relation
$$
DK_\eps(\theta)^\top\ J\circ K_\eps(\theta)\ DK_\eps(\theta)=0\ .
$$
The definitions~\ref{vectordiophantine} and \ref{lambdadiophantine} on the Diophantine property
of the frequency need to be modified as follows.

\begin{definition}\label{vectordiophantinef}
Let $\omega \in \real^d$, $\tau \in \real_+$.
We define the quantity $\nu(\omega;\tau)$ as
$$
\nu(\omega;\tau) \equiv \sup_ {k \in \integer^{d} \setminus \{0\}}
|\omega\cdot k|^{-1} |k|^{-\tau}\ .
$$
In the supremum above we allow infinity and we set $\nu(\omega;\tau) = \infty$,
whenever $|\omega\cdot k|=0$.

We say that  $\omega\in \real^d$ is Diophantine of
class $\tau$ and constant $\nu(\omega; \tau)$, if
\[
\nu(\omega; \tau) < \infty
\]
and we denote by $\D_d(\nu, \tau)$ the set of Diophantine vectors in $\real^d$ of class $\tau$ and
constant $\nu$.
\end{definition}

We also have the following definition.

\begin{definition} \label{lambdadiophantinef}
Let $\chi \in \complex$,  $\omega \in \real^d$, $\tau \in \real_+$.
We define the quantity $\nu(\chi; \omega, \tau)$ as
$$
\nu(\chi; \omega, \tau) \equiv
\sup_
{k \in \integer^{d} \setminus \{0\}}
|i\omega\cdot k+\chi |^{-1} |k|^{-\tau}\ .
$$
\end{definition}

We remark that, setting $\chi\equiv \chi_r+i\chi_i$, then we have the following inequality:
$$
\Big(|i\omega\cdot k+\chi |\ |k|^{\tau}\Big)^2 \geq |i\omega\cdot k+\chi |^2 \geq \chi_r^2+(\omega\cdot k+\chi_i)^2
\geq \chi_r^2\ ,
$$
from which we conclude that for $\chi_r\not=0$:
$$
\nu(\chi; \omega, \tau)\leq |\chi_r|^{-1}<\infty
$$
for any $\tau$. Therefore, the critical case becomes $\chi_r=0$, for which it could happen that
$\nu(\chi; \omega, \tau)=\infty$ (for example when $\chi_i=-\omega\cdot k$).

\vskip.1in

We assume that the function $\chi=\chi(\eps)$ is analytic in a neighborhood of $\eps=0$ and that
$\chi(0)=0$; this leads us to require that $\chi$ satisfies (compare with assumption {\bf H$\lambda$} for the
mapping case):

{\bf H$\bf{\chi}\qquad\qquad\qquad\qquad\qquad$} $\chi(\eps) = \alpha \eps^a + O(|\eps|^{a+1})$

\noindent
for some $a\in\integer_+$, $\alpha\in\complex$.

In analogy to \equ{goodset} and \equ{Lambdadefined}, we define the set $\G=\G(A; \omega,\tau, N,\chi)$
and its preimage under $\chi$, say $\Lambda=\Lambda(A; \omega, \tau, N)$, as follows:
\beqa{gf}
\G(A; \omega,\tau, N,\chi)&=&\{ \eps \in \complex : \quad \nu(\chi(\eps); \omega, \tau) |\chi(\eps)|^{N+1} \le A \}\nonumber\\
\Lambda(A; \omega, \tau, N)&=&\{ \chi \in  \complex  : \quad
\nu(\chi; \omega, \tau) |\chi(\eps)|^{N+1} \le A \}\ ,
\eeqa
where $A>0$, $N\in\integer_+$. Moreover, let
\beq{AGr0}
\G_{r_0}(A; \omega,\tau, N,\chi)  = \G \cap \{\eps\in\complex:\  |\eps| \le r_0\}\ .
\eeq

The equivalent of the cohomological equation \equ{twisted_cohomology} is given by
\beq{cohf}
\partial_\omega \varphi(\theta)+\chi \ \varphi(\theta)=\chi(\theta)\ ,
\eeq
where $\chi=\chi(\theta)$ is a known function with zero average. The solution of \equ{cohf} can be found
through a result extending Lemma~\ref{cohomology}, under the assumption
that $\omega$ is a Diophantine vector.

The automatic reducibility is ensured by the following result, which has been proven in \cite{CallejaCL11}.

\begin{proposition}\label{prop:flow1} (see \cite{CallejaCL11})
Let $N(\theta)\equiv(DK(\theta)^\top DK(\theta))^{-1}$ and let
$$
M(\theta)\equiv[DK(\theta)\,|\, JDK(\theta)N(\theta)]\ .
$$
Setting $A\equiv \nabla \F_{\mu_\eps,\eps}\circ K$, one has:
$$
\partial_\omega M(\theta)-A(\theta)M(\theta) = M(\theta)
\begin{pmatrix} 0 & S(\theta)\\ 0&\chi\, {\rm Id} \end{pmatrix}
$$
with
\beq{Sdeff}
S(\theta)\equiv N(\theta)DK(\theta)^\top J\Big(A(\theta)+A(\theta)^\top\Big) DK(\theta)N(\theta)\ .
\eeq
\end{proposition}
We now assume to start
with an approximate solution $(K_a,\mu_a)$, satisfying the invariance equation \equ{invf}:
$$
\partial_\omega K_a(\theta)-\F_{\mu_a,\eps}\circ K_a(\theta)=E(\theta)
$$
with an error term $E=E(\theta)$. Then, we can prove the equivalent of Theorem~\ref{mainKAM},
provided that the non--degeneracy condition is replaced by
$$
\det
\left(
\begin{array}{cc}
  {\overline S} & {\overline {S(B_{b})^0}}+\overline{\widetilde A_{1}} \\
  \chi\Id & \overline{\widetilde A_{2}} \\
 \end{array}%
\right) \ne 0\ ,
$$
where $S$ is defined in \equ{Sdeff}, $\widetilde A$ is given by
$$
\widetilde A=[\widetilde A_1|\widetilde A_2]=-M^{-1}\ (\nabla_\mu \F_{\mu_\eps,\eps}\circ K)
$$
and $(B_{b})^0$ solves the differential equation:
$$
\partial_\omega (B_{b})^0+\chi (B_{b})^0=-(\widetilde A_2)^0\ .
$$
For completeness, we conclude by stating the main theorem for the case of vector fields.

\begin{theorem}\label{mainf}
Let $\M\equiv\torus^d\times\B$ with $\B\subseteq\real^d$ an open, simply connected domain with smooth
boundary as in Theorem~\ref{main}.
Let $\omega \in \real^d$ satisfy Definition~\ref{vectordiophantinef} and let
$\F_{\mu,\eps}$  be a family of conformally symplectic vector fields with
conformal factor $\chi=\chi(\eps)$ as in {\bf H}$\chi$.
We assume that for $\eps = 0$, the vector field $\F_{\mu,0}$ is symplectic.

Assume that for $\mu=\mu_0$ the flow $\F_{\mu_0, 0}$ admits a Lagrangian invariant torus
with embedding $K_0 \in \A_\rho$, $\rho > 0$, satisfying
$$
\partial_\omega K_0(\theta)-\F_{\mu_0,\eps}\circ K_0(\theta)=0\ .
$$
Define the following quantities:
\beqano
N_0(\theta)&=&\Big(DK_0(\theta)^\top DK_0(\theta)\Big)^{-1}\nonumber\\
M_0(\theta)&=&[DK_0(\theta)\,|\, J\, DK_0(\theta)N_0(\theta)]\nonumber\\
\widetilde A_0&=&[\widetilde A_{10}|\widetilde A_{20}]=-M_0^{-1}\ (\nabla_\mu \F_{\mu_0,\eps}\circ K_0)\nonumber\\
A_0(\theta)&=&\nabla \F_{\mu_0,\eps}\circ K_0\nonumber\\
S_0(\theta)&=& N_0(\theta)DK_0(\theta)^\top J\Big(A_0(\theta)+A_0(\theta)^\top\Big)DK_0(\theta)N_0(\theta)\ .
\eeqano
Assume that the torus $K_0$ satisfies the following assumption.

{\bf HND'} Let the following non--degeneracy condition be satisfied:
$$
\det
\left(
\begin{array}{cc}
  {\overline S_0} & {\overline {S_0(B_{b0})^0}}+\overline{\widetilde A_{10}} \\
  0 & \overline{\widetilde A_{20}} \\
 \end{array}%
\right) \ne 0\ ,
$$
where $(B_{b0})^0$ is the solution of the cohomology equation
$$
\partial_\omega (B_{b0})^0+\chi (B_{b0})^0=-(\widetilde A_{20})^0\ .
$$

Then, we have the following result.

A) We can find a formal power series expansion $K_{\eps}^{[\infty]}$, $\mu_\eps^{[\infty]}$
satisfying \eqref{invf} in the sense that, expanding $K_{\eps}^{[\infty]}$, $\mu_\eps^{[\infty]}$ in power series
and taking a truncation to order $N\in\nat$ as
$K_{\eps}^{[\le N]} = \sum_{j=0}^{N}\eps^j  K_j$, $\mu_{\eps}^{[\le N]} = \sum_{j=0}^{N}\eps^j  \mu_j$,
then we have:
$$
|| \partial_\omega K^{[\le N]}_\eps -  \F_{\mu_{\eps}^{[\le N]},\eps}\circ K_\eps^{[\le N]}||_{\rho'} \le C_N |\eps|^{N+1}
$$
for some $0<\rho' < \rho$ and for some constant $C_N>0$.

B) For any $N \in \nat$, we can find a set $\G_{r_0}$ as in \equ{AGr0}
with $r_0$ sufficiently small and for any $0 < \rho' < \rho$ we can find
$K_\eps: \G_{r_0} \rightarrow \A_{\rho'}$, $\mu_\eps:  \G_{r_0} \rightarrow \complex^d$,
which are analytic in $\G_{r_0}$ and extend continuously to the boundary of $\G_{r_0}$,
such that for every $\eps\in\G_{r_0}$ they satisfy exactly the invariance equation:
$$
\partial_\omega K_\eps - \F_{\mu_\eps,\eps}\circ K_\eps=0\ .
$$
Moreover, the exact solution has the formal power series of part A) as an asymptotic expansion, in the sense that
\beqano
||K^{[\le N]}_\eps -  K_\eps||_{\rho'}  &\le& C_N |\eps|^{N+1} \ ,\nonumber\\
|\mu^{[\le N]}_\eps -  \mu_\eps|  &\le& C_N |\eps|^{N+1}\ .
\eeqano
\end{theorem}

The proof of Theorem~\ref{mainf} is similar to that of Theorem~\ref{main} and is left to the reader;
it relies on Theorem 20 in \cite{CallejaCL11}.

\def\cprime{$'$} \def\cprime{$'$} \def\cprime{$'$} \def\cprime{$'$}
  \def\cprime{$'$} \def\cprime{$'$}


\begin{thebibliography}{CCCdlL15}

\bibitem[AFC90]{ArtecaFC90}
G.~A. Arteca, F.~M. Fern{\'a}ndez, and E.~A. Castro.
\newblock {\em Large order perturbation theory and summation methods in quantum
  mechanics}, volume~53 of {\em Lecture Notes in Chemistry}.
\newblock Springer-Verlag, Berlin, 1990.

\bibitem[Agr10]{Agrachev}
A.~A. Agrach{\"e}v.
\newblock Invariant {L}agrangian submanifolds of dissipative systems.
\newblock {\em Uspekhi Mat. Nauk}, 65(5(395)):185--186, 2010.

\bibitem[AR67]{AbrahamR67}
Ralph Abraham and Joel Robbin.
\newblock {\em Transversal mappings and flows}.
\newblock W. A. Benjamin, Inc., New York-Amsterdam, 1967.

\bibitem[Arn61]{Arnold65}
V.~I. Arnol{\cprime}d.
\newblock Small denominators. {I}. {M}apping the circle onto itself.
\newblock {\em Izv. Akad. Nauk SSSR Ser. Mat.}, 25:21--86, 1961.

\bibitem[Bal94]{Balser}
Werner Balser.
\newblock {\em From divergent power series to analytic functions}, volume 1582
  of {\em Lecture Notes in Mathematics}.
\newblock Springer-Verlag, Berlin, 1994.
\newblock Theory and application of multisummable power series.

\bibitem[Ban02]{Banyaga02}
A.~Banyaga.
\newblock Some properties of locally conformal symplectic structures.
\newblock {\em Comment. Math. Helv.}, 77(2):383--398, 2002.

\bibitem[Ben88]{Bensoussan88}
Alain Bensoussan.
\newblock {\em Perturbation methods in optimal control}.
\newblock Wiley/Gauthier-Villars Series in Modern Applied Mathematics. John
  Wiley \& Sons Ltd., Chichester, 1988.
\newblock Translated from the French by C. Tomson.

\bibitem[BGM96]{BakerGM96}
George~A. Baker, Jr. and Peter Graves-Morris.
\newblock {\em Pad\'e approximants}, volume~59 of {\em Encyclopedia of
  Mathematics and its Applications}.
\newblock Cambridge University Press, Cambridge, second edition, 1996.

\bibitem[Bis09]{Biscani}
Francesco Biscani.
\newblock The {P}iranha algebraic manipulator.
\newblock {\em CoRR}, abs/0907.2076, 2009.

\bibitem[BK78]{Brent}
R.~P. Brent and H.~T. Kung.
\newblock Fast algorithms for manipulating formal power series.
\newblock {\em J. Assoc. Comput. Mach.}, 25(4):581--595, 1978.

\bibitem[Bor17]{Borel17}
\'Emile Borel.
\newblock {\em {Le\c cons sur les fonctions monog\`enes uniformes d'une
  variable complexe. R\'edig\'ees par {\it G. Julia}.}}
\newblock {Paris: Gauthier-Villars, XIl u. 165 S }, 1917.

\bibitem[Car52]{Caratheodory3}
C.~Carath{\'e}odory.
\newblock {\em Conformal representation}.
\newblock Cambridge Tracts in Mathematics and Mathematical Physics, no. 28.
  Cambridge, at the University Press, 1952.
\newblock 2d ed.

\bibitem[Car54a]{Caratheodory2}
C.~Carath{\'e}odory.
\newblock {\em Theory of functions of a complex variable. {V}ol. 1}.
\newblock Chelsea Publishing Co., New York, N. Y., 1954.
\newblock Translated by F. Steinhardt.

\bibitem[Car54b]{Caratheodory1}
C.~Caratheodory.
\newblock {\em Theory of functions of a complex variable. {V}ol. 2}.
\newblock Chelsea Publishing Company, New York, 1954.
\newblock Translated by F. Steinhardt.

\bibitem[CC09]{CellettiC09}
A.~Celletti and L.~Chierchia.
\newblock Quasi-periodic {A}ttractors in {C}elestial {M}echanics.
\newblock {\em Arch. Rational Mech. Anal.}, 191(2):311--345, 2009.

\bibitem[CC10]{CallejaC10}
R.~Calleja and A.~Celletti.
\newblock Breakdown of invariant attractors for the dissipative standard map.
\newblock {\em CHAOS}, 20(1):013121, 2010.

\bibitem[CCCdlL15]{CCCL2015}
R.~Calleja, A.~Celletti, L.~Corsi, and R.~de~la Llave.
\newblock Response solutions for quasi-periodically forced, dissipative wave
  equations.
\newblock {\em Preprint, {\tt
  http://www.ma.\-utexas.edu/mp\_arc/c/15/15-10.pdf}}, 2015.

\bibitem[CCdlL13a]{CallejaCL11b}
R.~Calleja, A.~Celletti, and R.~de~la Llave.
\newblock Local behavior near quasi-periodic solutions of conformally
  symplectic systems.
\newblock {\em Jour. Dyn. Diff. Eq.}, 55(3):821--841, 2013.

\bibitem[CCdlL13b]{CallejaCL13}
Renato~C. Calleja, Alessandra Celletti, and Rafael de~la Llave.
\newblock Construction of response functions in forced strongly dissipative
  systems.
\newblock {\em Discrete Contin. Dyn. Syst.}, 33(10):4411--4433, 2013.

\bibitem[CCdlL13c]{CallejaCL11}
Renato~C. Calleja, Alessandra Celletti, and Rafael de~la Llave.
\newblock A {KAM} theory for conformally symplectic systems: {E}fficient
  algorithms and their validation.
\newblock {\em J. Differential Equations}, 255(5):978--1049, 2013.

\bibitem[CCFdlL14]{CCFL14}
Renato~C. Calleja, Alessandra Celletti, Corrado Falcolini, and Rafael de~la
  Llave.
\newblock An {E}xtension of {G}reene's {C}riterion for {C}onformally
  {S}ymplectic {S}ystems and a {P}artial {J}ustification.
\newblock {\em SIAM J. Math. Anal.}, 46(4):2350--2384, 2014.

\bibitem[CdlL09]{CallejaL09}
R.~Calleja and R.~de~la Llave.
\newblock Fast numerical computation of quasi-periodic equilibrium states in
  1{D} statistical mechanics, including twist maps.
\newblock {\em Nonlinearity}, 22(6):1311--1336, 2009.

\bibitem[CdlL10a]{CallejaL10b}
R.~Calleja and R.~de~la Llave.
\newblock Computation of the breakdown of analyticity in statistical mechanics
  models: numerical results and a renormalization group explanation.
\newblock {\em Jour. of Stat. Phys}, 141(6):940--951, 2010.

\bibitem[CdlL10b]{CallejaL10}
R.~Calleja and R.~de~la Llave.
\newblock A numerically accessible criterion for the breakdown of
  quasi-periodic solutions and its rigorous justification.
\newblock {\em Nonlinearity}, 23(9):2029--2058, 2010.

\bibitem[Cel10]{Celletti2010}
Alessandra Celletti.
\newblock {\em Stability and chaos in celestial mechanics}.
\newblock Springer-Verlag, Berlin; published in association with Praxis
  Publishing Ltd., Chichester, 2010.

\bibitem[CF12]{CallejaF11}
R.~Calleja and J-Ll. Figueras.
\newblock Collision of invariant bundles of quasi-periodic attractors in the
  dissipative standard map.
\newblock {\em Chaos}, 23(021203), 2012.

\bibitem[CFG13]{CorsiFG13}
Livia Corsi, Roberto Feola, and Guido Gentile.
\newblock Domains of analyticity for response solutions in strongly dissipative
  forced systems.
\newblock {\em J. Math. Phys.}, 54(12):122701, 7, 2013.

\bibitem[CFG14]{CorsiFG14}
Livia Corsi, Roberto Feola, and Guido Gentile.
\newblock Convergent series for quasi-periodically forced strongly dissipative
  systems.
\newblock {\em Commun. Contemp. Math.}, 16(3):1350022, 20, 2014.

\bibitem[CGGG07]{CostinGGG07}
O.~Costin, G.~Gallavotti, G.~Gentile, and A.~Giuliani.
\newblock Borel summability and {L}indstedt series.
\newblock {\em Comm. Math. Phys.}, 269(1):175--193, 2007.

\bibitem[CMS14]{CarminatiMS14}
Carlo Carminati, Stefano Marmi, and David Sauzin.
\newblock There is only one {KAM} curve.
\newblock {\em Nonlinearity}, 27(9):2035--2062, 2014.

\bibitem[DFIZ14]{DFIZ2014}
A.~Davini, A.~Fathi, R.~Iturriaga, and M.~Zavidovique.
\newblock Convergence of the solutions of the discounted equation.
\newblock {\em Preprint, {\tt http://arxiv.org/pdf/1408.6712.pdf}}, 2014.

\bibitem[dlL01]{Llave01c}
R.~de~la Llave.
\newblock A tutorial on {K}{A}{M} theory.
\newblock In {\em Smooth ergodic theory and its applications (Seattle, WA,
  1999)}, pages 175--292. Amer. Math. Soc., Providence, RI, 2001.

\bibitem[dlLGJV05]{LlaveGJV05}
R.~de~la Llave, A.~Gonz{\'a}lez, {\`A}.~Jorba, and J.~Villanueva.
\newblock K{AM} theory without action-angle variables.
\newblock {\em Nonlinearity}, 18(2):855--895, 2005.

\bibitem[dlLO99]{LlaveO99}
R.~de~la Llave and R.~Obaya.
\newblock Regularity of the composition operator in spaces of {H}\"older
  functions.
\newblock {\em Discrete Contin. Dynam. Systems}, 5(1):157--184, 1999.

\bibitem[dlLO00]{LlaveO00}
R.~de~la Llave and R.~Obaya.
\newblock Decomposition theorems for groups of diffeomorphisms in the sphere.
\newblock {\em Trans. Amer. Math. Soc.}, 352(3):1005--1020, 2000.

\bibitem[dlLR91]{LlaveR90}
R.~de~la Llave and D.~Rana.
\newblock Accurate strategies for {K}.{A}.{M}.\ bounds and their
  implementation.
\newblock In {\em Computer Aided Proofs in Analysis (Cincinnati, OH, 1989)},
  pages 127--146. Springer, New York, 1991.

\bibitem[DM06]{DettmannM96}
C.~P. Dettmann and G.~P. Morris.
\newblock Proof of {L}yapunov exponent pairing for systems at constant kinetic
  energy.
\newblock {\em Phys. Rev. E}, 53(6):R5545--R5548, 2006.

\bibitem[Dua94]{Duarte1994}
Pedro Duarte.
\newblock Plenty of elliptic islands for the standard family of area preserving
  maps.
\newblock {\em Ann. Inst. H. Poincar\'e Anal. Non Lin\'eaire}, 11(4):359--409,
  1994.

\bibitem[Dua99]{Duarte1999}
Pedro Duarte.
\newblock Abundance of elliptic isles at conservative bifurcations.
\newblock {\em Dynam. Stability Systems}, 14(4):339--356, 1999.

\bibitem[Fef09]{Fefferman09}
Charles Fefferman.
\newblock Whitney's extension problems and interpolation of data.
\newblock {\em Bull. Amer. Math. Soc. (N.S.)}, 46(2):207--220, 2009.

\bibitem[GBD06]{GentileBD06}
Guido Gentile, Michele~V. Bartuccelli, and Jonathan H.~B. Deane.
\newblock Quasiperiodic attractors, {B}orel summability and the {B}ryuno
  condition for strongly dissipative systems.
\newblock {\em J. Math. Phys.}, 47(7):072702, 10, 2006.

\bibitem[Gen10]{Gentile10}
Guido Gentile.
\newblock Quasi-periodic motions in strongly dissipative forced systems.
\newblock {\em Ergodic Theory Dynam. Systems}, 30(5):1457--1469, 2010.

\bibitem[GG02]{GallavottiG02}
G.~Gallavotti and G.~Gentile.
\newblock Hyperbolic low-dimensional invariant tori and summations of divergent
  series.
\newblock {\em Comm. Math. Phys.}, 227(3):421--460, 2002.

\bibitem[GG05]{GentileG05}
Guido Gentile and Giovanni Gallavotti.
\newblock Degenerate elliptic resonances.
\newblock {\em Comm. Math. Phys.}, 257(2):319--362, 2005.

\bibitem[GGG06]{GGG2006}
G.~Gallavotti, G.~Gentile, and A.~Giuliani.
\newblock Fractional {L}indstedt series.
\newblock {\em J. Math. Phys.}, 47(1):012702, 33, 2006.

\bibitem[GM96]{GentileM96}
G.~Gentile and V.~Mastropietro.
\newblock Methods for the analysis of the {L}indstedt series for {K}{A}{M} tori
  and renormalizability in classical mechanics. {A} review with some
  applications.
\newblock {\em Rev. Math. Phys.}, 8(3):393--444, 1996.

\bibitem[Gra14]{Grafakos14}
Loukas Grafakos.
\newblock {\em Modern {F}ourier analysis}, volume 250 of {\em Graduate Texts in
  Mathematics}.
\newblock Springer, New York, third edition, 2014.

\bibitem[Har49]{Hardy}
Godfrey~H. Hardy.
\newblock {\em Divergent {S}eries}.
\newblock Oxford, at the Clarendon Press, 1949.

\bibitem[Har11]{Haro}
A.~Haro.
\newblock Automatic differentiation tools in computational dynamical systems.
\newblock {\em Work in progress}, 2011.

\bibitem[Her85]{Herman85b}
Michael-R. Herman.
\newblock Simple proofs of local conjugacy theorems for diffeomorphisms of the
  circle with almost every rotation number.
\newblock {\em Bol. Soc. Brasil. Mat.}, 16(1):45--83, 1985.

\bibitem[Hil48]{HilleP48}
Einar Hille.
\newblock {\em Functional {A}nalysis and {S}emi-{G}roups}.
\newblock American Mathematical Society Colloquium Publications, vol. 31.
  American Mathematical Society, New York, 1948.

\bibitem[ISM11]{IturriagaS11}
Renato Iturriaga and H{\'e}ctor S{\'a}nchez-Morgado.
\newblock Limit of the infinite horizon discounted {H}amilton-{J}acobi
  equation.
\newblock {\em Discrete Contin. Dyn. Syst. Ser. B}, 15(3):623--635, 2011.

\bibitem[JdlLZ99]{JorbaLZ00}
{\`A}.~Jorba, R.~de~la Llave, and M.~Zou.
\newblock {L}indstedt series for lower-dimensional tori.
\newblock In {\em {H}amiltonian Systems with Three or More Degrees of Freedom
  (S'Agar\'o, 1995)}, pages 151--167. Kluwer Acad. Publ., Dordrecht, 1999.

\bibitem[Kol57]{Kolmogorov57}
A.~N. Kolmogorov.
\newblock Th\'eorie g\'en\'erale des syst\`emes dynamiques et m\'ecanique
  classique.
\newblock In {\em Proceedings of the {I}nternational {C}ongress of
  {M}athematicians, {A}msterdam, 1954, {V}ol. 1}, pages 315--333. Erven P.
  Noordhoff N.V., Groningen; North-Holland Publishing Co., Amsterdam, 1957.

\bibitem[KSF01]{KleinertSF01}
Hagen Kleinert and Verena Schulte-Frohlinde.
\newblock {\em Critical properties of {$\phi^4$}-theories}.
\newblock World Scientific Publishing Co., Inc., River Edge, NJ, 2001.

\bibitem[LdlL10]{LiL10}
Xuemei Li and Rafael de~la Llave.
\newblock Convergence of differentiable functions on closed sets and remarks on
  the proofs of the ``converse approximation lemmas''.
\newblock {\em Discrete Contin. Dyn. Syst. Ser. S}, 3(4):623--641, 2010.

\bibitem[LGZJ90]{LeGZJ90}
J.C. Le~Guillou and J.~Zinn-Justin, editors.
\newblock {\em Large-order behaviour of perturbation theory}.
\newblock North Holland, 1990.

\bibitem[MS03]{MarmiS03}
S.~Marmi and D.~Sauzin.
\newblock Quasianalytic monogenic solutions of a cohomological equation.
\newblock {\em Mem. Amer. Math. Soc.}, 164(780):vi+83, 2003.

\bibitem[MS11]{MarmiS11}
Stefano Marmi and David Sauzin.
\newblock A quasianalyticity property for monogenic solutions of small divisor
  problems.
\newblock {\em Bull. Braz. Math. Soc. (N.S.)}, 42(1):45--74, 2011.

\bibitem[Nel69]{Nelson69}
E.~Nelson.
\newblock {\em Topics in Dynamics. {I}: {F}lows}.
\newblock Princeton University Press, Princeton, N.J., 1969.

\bibitem[PL08]{PhragmenL08}
Lars~Edvard Phragm\'en and Ernst Lindel\"of.
\newblock Sur une extension d'un principe classique de l'analyse et sur
  quelques propri\'et\'es des fonctions monog\`enes dans le voisinage d'un
  point singulier.
\newblock {\em Acta Math.}, 31(1), 1908.

\bibitem[Poi87]{Poincarefrench}
H.~Poincar{\'e}.
\newblock {\em Les m\'ethodes nouvelles de la m\'ecanique c\'eleste. {T}ome
  {II}}.
\newblock Les Grands Classiques Gauthier-Villars. [Gauthier-Villars Great
  Classics]. Librairie Scientifique et Technique Albert Blanchard, Paris, 1987.
\newblock M{\'e}thodes de MM. Newcomb, Gyld{\'e}n, Lindstedt et Bohlin. [The
  methods of Newcomb, Gyld{\'e}n, Lindstedt and Bohlin], Reprint of the 1893
  original, Biblioth{\`e}que Scientifique Albert Blanchard. [Albert Blanchard
  Scientific Library].

\bibitem[R{\"u}s75]{Russmann75}
H.~R{\"u}ssmann.
\newblock On optimal estimates for the solutions of linear partial differential
  equations of first order with constant coefficients on the torus.
\newblock In {\em Dynamical Systems, Theory and Applications (Battelle
  Rencontres, Seattle, Wash., 1974)}, pages 598--624. Lecture Notes in Phys.,
  Vol. 38, Berlin, 1975. Springer.

\bibitem[R{\"u}s76a]{Russmann76b}
H.~R{\"u}ssmann.
\newblock Note on sums containing small divisors.
\newblock {\em Comm. Pure Appl. Math.}, 29(6):755--758, 1976.

\bibitem[R{\"u}s76b]{Russmann76a}
H.~R{\"u}ssmann.
\newblock On optimal estimates for the solutions of linear difference equations
  on the circle.
\newblock {\em Celestial Mech.}, 14(1):33--37, 1976.

\bibitem[Sch80]{Schmidt80}
Wolfgang~M. Schmidt.
\newblock {\em Diophantine approximation}, volume 785 of {\em Lecture Notes in
  Mathematics}.
\newblock Springer, Berlin, 1980.

\bibitem[Sie54]{Siegel54}
C.~L. Siegel.
\newblock \"{U}ber die {E}xistenz einer {N}ormalform analytischer
  {H}amiltonscher {D}ifferentialgleichungen in der {N}\"ahe einer
  {G}leichgewichtsl\"osung.
\newblock {\em Math. Ann.}, 128:144--170, 1954.

\bibitem[SL12]{Locatelli}
Letizia Stefanelli and Ugo Locatelli.
\newblock {K}olmogorov's normal form for equations of motion with dissipative
  effects.
\newblock {\em Discrete Contin. Dynam. Systems}, 17(7):2561--2593, 2012.

\bibitem[SL15]{StefanelliL15}
Letizia Stefanelli and Ugo Locatelli.
\newblock Quasi-periodic motions in a special class of dynamical equations with
  dissipative effects: {A} pair of detection methods.
\newblock {\em Discrete Contin. Dyn. Syst. Ser. B}, 20(4):1155--1187, 2015.

\bibitem[Sok80]{Sokal}
Alan~D. Sokal.
\newblock An improvement of {W}atson's theorem on {B}orel summability.
\newblock {\em J. Math. Phys.}, 21(2):261--263, 1980.

\bibitem[Ste70]{Stein70}
E.~M. Stein.
\newblock {\em Singular integrals and differentiability properties of
  functions}.
\newblock Princeton University Press, Princeton, N.J., 1970.
\newblock Princeton Mathematical Series, No. 30.

\bibitem[SZ65]{SaksZ65}
Stanis{\l}aw Saks and Antoni Zygmund.
\newblock {\em Analytic functions}.
\newblock Second edition, enlarged. Translated by E. J. Scott. Monografie
  Matematyczne, Tom 28. Pa\'nstwowe Wydawnietwo Naukowe, Warsaw, 1965.

\bibitem[Whi36]{Whitney36}
H.~Whitney.
\newblock Differentiable functions defined in arbitrary subsets of {E}uclidean
  space.
\newblock {\em Trans. Amer. Math. Soc.}, 40(2):309--317, 1936.

\bibitem[Win93]{Winkler93}
J{\"o}rg Winkler.
\newblock A uniqueness theorem for monogenic functions.
\newblock {\em Ann. Acad. Sci. Fenn. Ser. A I Math.}, 18(1):105--116, 1993.

\bibitem[WL98]{WojtkowskiL98}
M.~P. Wojtkowski and C.~Liverani.
\newblock Conformally symplectic dynamics and symmetry of the {L}yapunov
  spectrum.
\newblock {\em Comm. Math. Phys.}, 194(1):47--60, 1998.

\end{thebibliography}
\end{document}